\documentclass[12pt]{article}

\usepackage{amssymb}
\usepackage{amsmath}
\usepackage{amsthm}
\usepackage{color}
\usepackage{caption}
\usepackage{graphics}
\usepackage{enumerate}
\usepackage{epsfig,latexsym,verbatim}
\usepackage{graphicx}
\usepackage{multirow}
\usepackage{multicol}
\usepackage{float}
\usepackage{hyperref}
\usepackage{graphicx}
\usepackage{amsfonts}
\usepackage[authoryear,sort]{natbib}
\usepackage{booktabs}
\usepackage{bm}
\usepackage{subcaption}
\usepackage{accents}
\usepackage[symbol]{footmisc}
\usepackage[flushleft]{threeparttable}

\allowdisplaybreaks

\numberwithin{equation}{section}

\newtheorem{theorem}{\bf Theorem}
\newtheorem{lemma}{\bf Lemma}
\newtheorem{remark}{\bf Remark}
\let\hat\widehat
\let\tilde\widetilde

\DeclareMathOperator*{\argmin}{\arg \min}

\theoremstyle{definition}
\newtheorem{condition}{Condition}

\newtheorem{exam}{Example}[section]
\newtheorem*{excont}{Example \continuation}
\newcommand{\continuation}{??}
\newenvironment{continueexample}[1]
{\renewcommand{\continuation}{\ref{#1}}\excont[continued]}
{\endexcont}

\def\tit.arg{{\bf On Asymptotic Optimality of Least Squares Model Averaging When True Model Is Included\footnote{We thank Fang Fang and Yuhong Yang for their valuable comments and suggestions at the “2024 Conference on the Frontiers of Model Averaging, Prediction Theory, and Machine Learning”. Xu's research was partially supported by National Natural Science Foundation of China (12101591). Zhang's research was partially supported by the National Natural Science Foundation of China (71925007, 72091212, and 71988101), Beijing Natural Science Foundation (Z240004), and the CAS Project for Young Scientists in Basic Research (YSBR-008).}}
}

\def\key.arg{
}

\def\author.arg{
Wenchao Xu$^1$ and Xinyu Zhang$^2$ \\
$^1$School of Statistics and Information, Shanghai University of International Business and Economics, Shanghai, China\\
$^2$Academy of Mathematics and Systems Science, Chinese Academy of Sciences, Beijing, China
}
\def\abst.arg{

}

\begin{document}
\baselineskip=18pt

\thispagestyle{plain}
\thispagestyle{empty} 


\begin{center}
{\Large \tit.arg}

\vskip 6mm

\author.arg
\end{center}

\vskip 3mm
\centerline{\small ABSTRACT}
Asymptotic optimality is a key theoretical property in model averaging. Due to technical difficulties, existing studies rely on restricted weight sets or the assumption that there is no true model with fixed dimensions in the candidate set. The focus of this paper is to overcome these difficulties. Surprisingly, we discover that when the penalty factor in the weight selection criterion diverges with a certain order and the true model dimension is fixed, asymptotic loss optimality does not hold, but asymptotic risk optimality does. This result differs from the corresponding result of \citet[\textit{Econometric Theory} 39, 412--441]{fang2023} and reveals that using the discrete weight set of \citet[\textit{Econometrica} 75, 1175--1189]{hansen2007mma} can yield opposite asymptotic properties compared to using the usual weight set. Simulation studies illustrate the theoretical findings in a variety of settings.

\abst.arg


\setcounter{page}{1}
\newpage

\section{Introduction}\label{sec:intro}
\baselineskip=23pt

Least squares model averaging has attracted much attention in both econometrics and statistics since the seminal work on Mallows model averaging (MMA) by \citet{hansen2007mma}. Subsequently, numerous alternative methods have emerged to select the model averaging weights. These include optimal mean squared error averaging \citep{liang2011}, jackknife model averaging \citep{hansen2012,zhang2013JMA,ando2014}, heteroskedasticity-robust $C_p$ model averaging \citep{liu2013ma}, prediction model averaging \citep{xie2015}, Kullback-Leibler model averaging \citep{zhang2015}, parsimonious model averaging \citep[PMA,][]{zhang2020PMA}, and so on.

Asymptotic optimality is a key theoretical goal pursued in model averaging research. It states that model averaging estimator yields the smallest possible prediction loss (or risk) among all such estimators as the sample size approaches infinity. Most of the aforementioned model averaging methods have been shown to be asymptotically optimal under some conditions. To illustrate, taking MMA as an example. \citet{hansen2007mma} were the first to establish MMA's asymptotic optimality when the candidate models are nested and the weights are restricted to a discrete set. Subsequently, \citet{wan2010} and \citet{zhang2021} provided justification in a non-nested model setting with continuous weights. However, their works impose some conditions to limit the number of candidate models. More recently, \citet{peng.opt} significantly improved these conditions by allowing more candidate models to be combined in a nested model setting. It is important to note that these results typically require that all candidate models with fixed dimensions are approximations, which is satisfied in the scenario where all candidate models are misspecified.

In the area of model averaging, there is another common scenario where the true model is among the candidate models. This scenario is also frequently encountered in model selection. In this context, existing works, such as \citet{zhang2015el}, \citet{zhang.liu2019}, \citet{fang2020}, and \citet{zhang2020PMA}, mainly focus on the asymptotic behavior of the selected weight and the asymptotic distribution of the model averaging estimator. However, theoretical results on asymptotic optimality in this setting have received limited attention in the literature. Recently, \citet{fang2023} made a significant contribution to filling this gap in the nested model setting. To the best of our knowledge, this is the sole work addressing this topic. A related work by \citet{shao1997} established the asymptotic loss efficiency of various model selection procedures, such as Mallows' $C_p$, AIC, BIC, and cross-validation, in the scenario where a true model exists.

Let us consider a weight selection criterion \eqref{eq:crit} with a penalty factor $\phi_n$ defined in Section \ref{sec:setup}. This criterion encompasses MMA ($\phi_n=2$) and PMA ($\phi_n\to \infty$ or $\phi_n=\log n$) as special cases. When considering the scenario where the true model is among the candidate models, the findings of \citet{fang2023} can be summarized as follows. When $\phi_n=\log n$, the least squares model averaging is asymptotically loss optimal by putting weight one to the true model. When $\phi_n=2$ (MMA) and the model dimensions are fixed, asymptotic loss optimality does not hold since the model averaging method puts weight on the over-fitted models with a positive probability. These results are very similar to the asymptotic behaviors of the traditional model selection methods AIC and BIC \citep{shao1997,ding2018}. However, the aforementioned results of \citet{fang2023} restrict the model weights to a slightly special weight set due to technical difficulty. While they believed that these results remain valid for the general weight set, a formal theoretical justification is still lacking.

In this paper, we revisit the asymptotic optimality of least square model averaging in the nested model setting when the true model is among the candidate models. Compared to the results of \citet{fang2023}, our main contributions are twofold. First, our results hold for the general weight set, rather than their restricted weight set. Second, we explore asymptotic risk optimality, a topic that they did not address. We show that when $\phi_n=\log n$ and the true model dimension is fixed, the asymptotic loss optimality does not hold, which differs from the corresponding result of \citet{fang2023}. This also reveals that using the discrete weight set of \citet{hansen2007mma} for model averaging can lead to opposite asymptotic properties compared to using the usual weight set. When $\phi_n=\log n$ and the true model dimension does not diverge too fast, the asymptotic risk optimality holds. Besides, MMA with fixed model dimensions is neither asymptotically loss optimal nor asymptotically risk optimal, which coincides with the corresponding result of \citet{fang2023}.

The remainder of the paper is organized as follows. Section \ref{sec:setup} introduces a class of least squares model averaging methods, which includes MMA and PMA as special cases. Section \ref{sec:existres} reviews the existing results and provides a discussion. Section \ref{sec:mainres} presents the main theoretical results. Section \ref{sec:simu} provides the results of finite sample simulation studies. Section \ref{sec:con} concludes the paper. The proofs of the main results are relegated to the Appendix.

\section{Least Squares Model Averaging}\label{sec:setup}

Suppose we have $n$ independent observations $\{(y_i, \mathbf{x}_i)\colon i=1,\ldots,n\}$, where $y_i$ is the scalar response and $\mathbf{x}_i=(x_{i1},x_{i2},\ldots)$ is a vector of countably infinite covariates. Consider the linear regression model
\begin{equation}\label{eq:model}
	y_i=\mu_i+e_i=\sum_{j=1}^{\infty} \beta_jx_{ij}+e_i,
\end{equation}
where $e_1,\ldots, e_n$ are independent and identically distributed (i.i.d.) errors with $E(e_i|\mathbf{x}_i)=0$ and $E(e_i^2|\mathbf{x}_i)=\sigma^2$ and $\beta_j$'s are unknown parameters. In matrix notation, \eqref{eq:model} can be written as $\mathbf{y}=\bm{\mu}+\mathbf{e}$, where $\mathbf{y}=(y_1,\ldots,y_n)^{\top}$, $\bm{\mu}=(\mu_1,\ldots,\mu_n)^{\top}$, and $\mathbf{e}=(e_1,\ldots,e_n)^{\top}$.

To estimate the true mean vector $\bm{\mu}$, we consider $M_n$ nested candidate models for model averaging, where the $m$th model uses the first $k_m$ covariates and $0<k_1<\cdots<k_{M_n}<\infty$. Both $k_{M_n}$ and $M_n$ can be diverging to infinity as $n\to \infty$. Let $\mathbf{X}_m$ be the $n \times k_m$ design matrix of the $m$th model. We assume that $\mathbf{X}_m$ is of full column rank for any $m \in \{1,\ldots,M_n\}$. Then, under the $m$th model, the estimator of $\bm{\mu}$ is $\hat{\bm{\mu}}_m=\mathbf{P}_m \mathbf{y}$, where $\mathbf{P}_m=\mathbf{X}_m(\mathbf{X}_m^{\top}\mathbf{X}_m)^{-1}\mathbf{X}_m^{\top}$. Let $\mathbf{w}=(w_1,\ldots,w_{M_n})^{\top}$ be a weight vector belonging to the unit simplex in $\mathbb{R}^{M_n}$:
\begin{equation*}
	\mathcal{H}_n=\left\{\mathbf{w}\in [0, 1]^{M_n}\colon \sum_{m=1}^{M_n} w_m=1\right\}.
\end{equation*}
Then, the model averaging estimator of $\bm{\mu}$ with weights $\mathbf{w}$ is $\hat{\bm{\mu}}(\mathbf{w})=\sum_{m=1}^{M_n} w_m \hat{\bm{\mu}}_m=\mathbf{P}(\mathbf{w})\mathbf{y}$, where $\mathbf{P}(\mathbf{w})=\sum_{m=1}^{M_n} w_m\mathbf{P}_m$. Let $\mathbf{I}_n$ be the identity matrix of size $n$ and $\|\cdot\|$ be the Euclidean norm. Following \citet{zhang2020PMA}, the weight vector $\mathbf{w}$ is selected by minimizing the criterion
\begin{equation}\label{eq:crit}
	\mathcal{G}_n(\mathbf{w})=\|\{\mathbf{I}_n-\mathbf{P}(\mathbf{w})\}\mathbf{y}\|^2+\phi_n\hat{\sigma}^2 \mathbf{w}^{\top}\mathbf{K},
\end{equation}
where $\hat{\sigma}^2=(n-k_{M_n})^{-1}\mathbf{y}^{\top}(\mathbf{I}_n-\mathbf{P}_{M_n})\mathbf{y}$ is an estimator of $\sigma^2$, $\mathbf{K}=(k_1,\ldots,k_{M_n})^{\top}$, and $\phi_n$ is a penalty factor which may depend on $n$. Two common choices for $\phi_n$ are $\phi_n=2$ and $\phi_n=\log n$. When $\phi_n=2$, \eqref{eq:crit} is the Mallows criterion of \citet{hansen2007mma}. When all elements of $\mathbf{w}$ are 0 or 1, the weight selection criterion \eqref{eq:crit} corresponds to the AIC if $\phi_n=2$ and the BIC if $\phi_n=\log n$. Note that \citet{fang.EL2022} proposed a cross-validation procedure for selecting $\phi_n$ between $2$ and $\log n$. Finally, the selected weight vector is defined as $\hat{\mathbf{w}}=(\hat{w}_1,\ldots,\hat{w}_{M_n})^{\top}=\argmin_{\mathbf{w}\in \mathcal{H}_n} \mathcal{G}_n(\mathbf{w})$.

Define the squared prediction loss as $L_n(\mathbf{w})=\|\bm{\mu}-\hat{\bm{\mu}}(\mathbf{w})\|^2$ and the corresponding prediction risk as $R_n(\mathbf{w})=E\{L_n(\mathbf{w})|\mathbf{x}_1,\ldots,\mathbf{x}_n\}$. In the context of model averaging research, two types of asymptotic optimality are commonly considered: asymptotic loss and risk optimality. Specifically, we say that $\hat{\mathbf{w}}$ is asymptotically loss optimal if
\begin{equation}\label{eq:loss}
	\frac{L_n(\hat{\mathbf{w}})}{\inf_{\mathbf{w}\in \mathcal{H}_n} L_n(\mathbf{w})}\to_p 1, \tag{OPT.1}
\end{equation}
and we refer to $\hat{\mathbf{w}}$ as asymptotically risk optimal if
\begin{equation}\label{eq:risk}
	\frac{R_n(\hat{\mathbf{w}})}{\inf_{\mathbf{w}\in \mathcal{H}_n} R_n(\mathbf{w})}\to_p 1. \tag{OPT.2}
\end{equation}
Here $\to_p$ denotes convergence in probability. All limiting processes in this paper are with respect to $n\to \infty$. Note that \eqref{eq:loss} means that the squared prediction loss of the model averaging estimator is asymptotically as small as the squared prediction loss of the infeasible best possible averaging estimator, and \eqref{eq:risk} is a similar statement about the squared prediction risk.

When $\phi_n=2$ (MMA), \citet{hansen2007mma} established the asymptotic loss optimality of $\hat{\mathbf{w}}$. However, this requires weights to be confined to the discrete set
\begin{equation*}
	\mathcal{H}_n(N)=\left\{\mathbf{w}\colon w_m\in \left\{0,\frac{1}{N},\frac{2}{N},\ldots,1\right\},\sum_{m=1}^{M_n} w_m=1\right\},
\end{equation*}
where $N$ is a fixed positive integer. Note that model averaging restricted to $\mathcal{H}_n(1)$ reduces to model selection. Subsequently, \citet{wan2010}, \citet{zhang2021}, and \citet{peng.opt} demonstrated that $\hat{\mathbf{w}}$ is asymptotically loss optimal over the continuous weight set $\mathcal{H}_n$. Their proofs actually imply the asymptotic risk optimality of $\hat{\mathbf{w}}$ as well. When $\phi_n\to \infty$, \citet{zhang2020PMA} justified both the asymptotic loss and risk optimality of $\hat{\mathbf{w}}$. A key prerequisite for these results is that $\xi_n=\inf_{\mathbf{w}\in \mathcal{H}_n} R_n(\mathbf{w})\to \infty$. This condition means that there are no candidate models with fixed dimensions for which the approximation error is zero; see also \citet{hansen2007mma}. Therefore, the aforementioned results do not apply to the case where there exists a true candidate model with fixed dimension.

Our aim is to investigate both the asymptotic loss and risk optimality of $\hat{\mathbf{w}}$ in the scenario where the true model is included in the set of candidate models. Let $\mathcal{A}_m=\{1,\ldots,k_m\}$ and $\mathcal{A}_m^c$ denote its complement. We assume that there exists an $M_0$ such that $(\beta_j, j\in \mathcal{A}_{M_0+1}\cap \mathcal{A}_{M_0}^c)^{\top}\neq \mathbf{0}$ and $\beta_j=0$ for all $j>k_{M_0+1}$. Consequently, the first $M_0$ candidate models are under-fitted, the $(M_0+1)$th model is defined as the ``true model", and all other models are over-fitted. Additionally, we assume that $M_0\geq 1$ and $M_n>M_0+1$, meaning there is at least one under-fitted and one over-fitted model. Note that $k_{M_0+1}$ is the true model dimension, and we allow for diverging $k_{M_0+1}$ and $M_0$.

Throughout this paper, we consider the case of deterministic covariates. When $\mathbf{x}_1,\ldots, \mathbf{x}_n$ are random, the results in this article are still valid in the almost sure sense, provided that the required conditions involving $\mathbf{x}_1,\ldots, \mathbf{x}_n$ hold almost surely.

\begin{remark}\label{rem:scenarios1}
	In fact, existing works, such as \citet{wan2010}, \citet{zhang2020PMA}, \citet{zhang2021}, and \citet{peng.opt}, does not completely exclude the scenario in which the true model is among the candidate models. To illustrate this, when the true model dimension diverges, it follows from Lemma \ref{lem:risk} in Section \ref{sec:mainres} that $\xi_n\to \infty$ under certain conditions.  Therefore, the results from the aforementioned works could be applicable to our issue in the context of a diverging true model dimension. See Subsection \ref{subsec:summ} for more detailed discussions.
\end{remark}

\section{Existing Results and Discussion}\label{sec:existres}

\subsection{Review of the Existing Results}\label{subsec:review}

In a recent study, \citet{fang2023} established conditions that determine whether $\hat{\mathbf{w}}$ achieves the asymptotic loss optimality \eqref{eq:loss} when the true model is among the candidate models; see their Subsection 3.2 for the details. However, due to technical difficulty, their results rely on the following slightly special weight set
\begin{equation*}
	\mathcal{H}_n^{\delta}=\left\{\mathbf{w}\in \mathcal{H}_n\colon \sum_{m<M_0+1} w_m=0~\text{or}~\sum_{m<M_0+1} w_m\geq \delta n^{-\tau_0}\right\}.
\end{equation*}
Here, $\delta$ is a positive constant which can be arbitrarily small, and $\tau_0$ is a positive constant that satisfies that $\sum_{j\in \mathcal{A}_{M_0+1}\cap \mathcal{A}_{M_0}^c} \beta_j^2\geq c_{\tau_0} n^{-\tau_0}$ for a positive constant $c_{\tau_0}$. Note that when $\delta$ is sufficiently small or $n$ is large, the set $\mathcal{H}_n^{\delta}$ includes the discrete weight set $\mathcal{H}_n(N)$ considered by \citet{hansen2007mma} and \citet{hansen2012}.

By an abuse of notation, in this subsection, we still use $\hat{\mathbf{w}}$ to denote $\argmin_{\mathbf{w}\in \mathcal{H}_n^{\delta}} \mathcal{G}_n(\mathbf{w})$. The results of \citet{fang2023} are provided in their Theorem 5. Here, we summarize them as follows.

\begin{enumerate}[(i)]
	\item When $\phi_n=\log n$, as long as the true model dimension does not diverge too fast, $\hat{\mathbf{w}}$ is asymptotically loss optimal in the sense that
	\begin{equation*}
		\frac{L_n(\hat{\mathbf{w}})}{\inf_{\mathbf{w}\in \mathcal{H}_n^{\delta}} L_n(\mathbf{w})}\to_p 1
	\end{equation*}
	by pushing all the weights to the true model.
	\item When $\phi_n=2$ and the model dimensions are fixed, asymptotic loss optimality does not hold in the sense that
	\begin{equation*}
		\frac{L_n(\hat{\mathbf{w}})}{\inf_{\mathbf{w}\in \mathcal{H}_n^{\delta}} L_n(\mathbf{w})}\nrightarrow_p 1
	\end{equation*}
	 since the model averaging method puts weights on the over-fitted models with a positive probability.
\end{enumerate}
In the area of model selection, it is well known that in a parametric framework, AIC is not asymptotically loss optimal, whereas BIC is due to its consistency \citep{shao1997,ding2018}. Consequently, Results (i) and (ii) are very similar to the corresponding asymptotic behaviors of BIC and AIC, respectively.

Denote by $\mathbf{w}_m^0$ an $M_n\times 1$ vector in which the $m$th element is one and the others are zeros. Let $\mathbf{w}^*=\argmin_{\mathbf{w}\in \mathcal{H}_n^{\delta}} L_n(\mathbf{w})$. The proof of the above results of \citet{fang2023} is divided into the following two steps.
\begin{enumerate}[\bf Step 1.]
	\item Prove that $\Pr(\mathbf{w}^*=\mathbf{w}_{M_0+1}^0)\to 1$.
	\item Analyze the asymptotic behavior of $L_n(\hat{\mathbf{w}})/L_n(\mathbf{w}_{M_0+1}^0)$.
\end{enumerate}
Steps 1 and 2 correspond to the results in their Lemma 1 and Theorem 5, respectively. The special weight set $\mathcal{H}_n^{\delta}$ is mainly utilized in Step 1. If $\Pr(\mathbf{w}^*=\mathbf{w}_{M_0+1}^0)\to 1$, we have $\Pr\{\inf_{\mathbf{w}\in \mathcal{H}_n^{\delta}} L_n(\mathbf{w})=L_n(\mathbf{w}_{M_0+1}^0)\}\to 1$, and thus $L_n(\hat{\mathbf{w}})/\inf_{\mathbf{w}\in \mathcal{H}_n^{\delta}} L_n(\mathbf{w})$ and $L_n(\hat{\mathbf{w}})/L_n(\mathbf{w}_{M_0+1}^0)$ have the same asymptotic behavior. Finally, after finishing Step 1, the subsequent work is Step 2, which relies on the asymptotic properties of $\hat{w}_m$ for the under-fitted and over-fitted models. These properties have been extensively studied in the literature such as \citet{zhang.liu2019} and \citet{zhang2020PMA}.

\subsection{Discussion on the Existing Results}\label{subsec:discu}

Let $\mathbf{w}^L=\argmin_{\mathbf{w}\in \mathcal{H}_n} L_n(\mathbf{w})$ denote the optimal weight vector in terms of prediction loss. Considering that the aforementioned results of \citet{fang2023} are restricted to the special weight set $\mathcal{H}_n^{\delta}$, two natural questions arise:
\begin{itemize}
	\item[Q1.] Does their Lemma 1 remain valid when $\mathcal{H}_n^{\delta}$ is replaced by $\mathcal{H}_n$? That is, does $\Pr(\mathbf{w}^L=\mathbf{w}_{M_0+1}^0)\to 1$ hold?
	\item[Q2.] Does their Theorem 5, i.e., Results (i) and (ii), still hold when $\mathcal{H}_n^{\delta}$ is replaced by $\mathcal{H}_n$?
\end{itemize}

As discussed by \citet{fang2023}, intuitively, Questions Q1 and Q2 should have affirmative answers. Unfortunately, this assertion may not hold true. In the following, we present a toy example to illustrate this.

\begin{exam}\label{exam:toy}
Consider the linear regression model $y_i=\beta_1 x_{i1}+\beta_2 x_{i2}+\beta_3 x_{i3}+e_i$, where $e_i\sim N(0, \sigma^2)$. Assume that $\beta_2\neq 0$, $\beta_3=0$, and $\mathbf{K}=(1, 2, 3)^{\top}$. Therefore, $M_n=3$, $M_0=1$, and the second model is the true model. By Lemma \ref{lem:lossrisk} in the Appendix, we can rewrite $L_n(\mathbf{w})$ as
\begin{equation}\label{eq:newl}
	L_n(\mathbf{w})=L_n(\mathbf{w}_2^0)+w_1^2 \mathbf{y}^{\top}(\mathbf{P}_2-\mathbf{P}_1)\mathbf{y}-2w_1 \mathbf{y}^{\top}(\mathbf{P}_2-\mathbf{P}_1)\mathbf{e}+w_3^2 \mathbf{e}^{\top}(\mathbf{P}_3-\mathbf{P}_2)\mathbf{e}.
\end{equation}
Furthermore, we assume that the covariates are orthonormal that satisfy $\sum_{i=1}^n x_{ij}^2=n$ and $\sum_{i=1}^n x_{ij}x_{ik}=0$ when $j\neq k$. Then,
\begin{equation*}
	\mathbf{y}^{\top}(\mathbf{P}_2-\mathbf{P}_1)\mathbf{y}=\frac{1}{n}\left(n\beta_2+\sum_{i=1}^n x_{i2}e_i\right)^2\quad\text{and}\quad \mathbf{y}^{\top}(\mathbf{P}_2-\mathbf{P}_1)\mathbf{e}=\frac{1}{n}\left(\sum_{i=1}^n x_{i2}e_i\right)^2+\beta_2 \sum_{i=1}^n x_{i2}e_i.
\end{equation*}
Write $\mathbf{w}^L=(w_1^L,w_2^L,w_3^L)^{\top}$. By minimizing \eqref{eq:newl} over $\mathbf{w}\in \mathcal{H}_n$, we have
\begin{equation*}
	w_1^L=\begin{cases}
		0, & \text{if}~w_1^{\mathrm{opt}}\leq 0, \\
		w_1^{\mathrm{opt}}, & \text{if}~0<w_1^{\mathrm{opt}}<1, \\
		1, & \text{if}~w_1^{\mathrm{opt}}\geq 1,
	\end{cases}
\end{equation*}
$w_2^L=1-w_1^L$, and $w_3^L=0$, where
\begin{equation*}
	w_1^{\mathrm{opt}}=\frac{\mathbf{y}^{\top}(\mathbf{P}_2-\mathbf{P}_1)\mathbf{e}}{\mathbf{y}^{\top}(\mathbf{P}_2-\mathbf{P}_1)\mathbf{y}}=\frac{n^{-1}(\sum_{i=1}^n x_{i2}e_i)^2+\beta_2 \sum_{i=1}^n x_{i2}e_i}{n^{-1}(n\beta_2+\sum_{i=1}^n x_{i2}e_i)^2}.
\end{equation*}
Observe that $n^{-1/2}\sum_{i=1}^n x_{i2}e_i\sim N(0, \sigma^2)$. Therefore,
\begin{align}\label{eq:prob}
	\Pr(\mathbf{w}^L=\mathbf{w}_2^0)&=\Pr(w_2^L=1)=\Pr(w_1^{\mathrm{opt}}\leq 0) \notag \\
	&=\Pr\left\{\frac{1}{n}\left(\sum_{i=1}^n x_{i2}e_i\right)^2+\beta_2 \sum_{i=1}^n x_{i2}e_i\leq 0\right\} \notag \\
	&=\Pr(0\leq N(0, \sigma^2)\leq \sqrt{n}|\beta_2|) \notag \\
	&\to \frac{1}{2}.
\end{align}
As a result, $\Pr(\mathbf{w}^L=\mathbf{w}_2^0)\nrightarrow 1$ and $\Pr\{\inf_{\mathbf{w}\in \mathcal{H}_n} L_n(\mathbf{w})=L_n(\mathbf{w}_2^0)\}\nrightarrow 1$, which means that Lemma 1 of \citet{fang2023} does not hold when replacing $\mathcal{H}_n^{\delta}$ with $\mathcal{H}_n$. This provides a negative response to Question Q1 in this specific example.

Next, we investigate the asymptotic loss optimality of $\hat{\mathbf{w}}$ when $\phi_n=\log n$. Observe that
\begin{equation*}
	\frac{L_n(\hat{\mathbf{w}})}{\inf_{\mathbf{w}\in \mathcal{H}_n} L_n(\mathbf{w})}=\frac{L_n(\hat{\mathbf{w}})}{L_n(\mathbf{w}_2^0)}\times \frac{L_n(\mathbf{w}_2^0)}{\inf_{\mathbf{w}\in \mathcal{H}_n} L_n(\mathbf{w})}.
\end{equation*}
Under Condition \ref{cond:sigma} in the next section, Lemmas 1--2 of \citet{zhang2020PMA} established that when $\phi_n\to \infty$, $\hat{w}_1=O_p(\phi_n/n)$ and $\Pr(\hat{w}_3=0)\to 1$. This result along with \eqref{eq:newl} and the fact that $\mathbf{y}^{\top}(\mathbf{P}_2-\mathbf{P}_1)\mathbf{y}=O_p(n)$, $\mathbf{y}^{\top}(\mathbf{P}_2-\mathbf{P}_1)\mathbf{e}=O_p(1)$, and $L_n(\mathbf{w}_2^0)=\mathbf{e}^{\top}\mathbf{P}_2\mathbf{e}\sim \sigma^2 \chi^2_2$, yields that $L_n(\hat{\mathbf{w}})/L_n(\mathbf{w}_2^0)=1+O_p(\phi_n^2/n)$. Consequently, when $\phi_n=\log n$, the asymptotic loss optimality \eqref{eq:loss} holds if and only if $L_n(\mathbf{w}_2^0)/\inf_{\mathbf{w}\in \mathcal{H}_n} L_n(\mathbf{w})\to_p 1$. In Appendix \ref{apped:examp1}, we show that
\begin{equation}\label{eq:exmp1}
	\frac{\inf_{\mathbf{w}\in \mathcal{H}_n} L_n(\mathbf{w})}{L_n(\mathbf{w}_2^0)}\to_d U+(1-U) \mathrm{Beta}\left(\frac{1}{2},\frac{1}{2}\right),
\end{equation}
where $\to_d$ denotes convergence in distribution, $\mathrm{Beta}(a, b)$ is a random variable having Beta distribution with shape parameters $a, b>0$, and $U$ is a Bernoulli random variable with $\Pr(U=0)=\Pr(U=1)=1/2$ and is independent of $\mathrm{Beta}(1/2, 1/2)$. From \eqref{eq:exmp1}, it is clear that $L_n(\mathbf{w}_2^0)/\inf_{\mathbf{w}\in \mathcal{H}_n} L_n(\mathbf{w}) \nrightarrow_p 1$. In conclusion, when $\phi_n=\log n$, $\hat{\mathbf{w}}$ is not asymptotically loss optimal, i.e., \eqref{eq:loss} does not hold. This partially provides a negative response to Question Q2 in this example.
\end{exam}

We also conduct a simulation study to verify \eqref{eq:prob} and \eqref{eq:exmp1} in Example \ref{exam:toy}. We simulate $10000$ data sets from the model in Example \ref{exam:toy}, where $(\beta_1, \beta_2, \beta_3)=(-1, 0.1, 0)$, $x_{ij}$ are independently generated from $N(0, 1)$, and $e_i$ are i.i.d. from $N(0,1)$ and are independent of $x_{ij}$'s. Figure \ref{fig:exm1}(a) shows the simulated values of $\Pr(\mathbf{w}^L=\mathbf{w}_2^0)$ from $n=100$ to $2000$ with an increment of $100$. These values converge to around $0.5$ when $n\geq 1000$, aligning with \eqref{eq:prob}. Figure \ref{fig:exm1}(b) displays the kernel density estimates for the values of $L_n(\mathbf{w}^L)/L_n(\mathbf{w}_2^0)$ after excluding 1, for $n=100$, $1000$, and $2000$. The density estimates for $n=1000$ and $2000$ closely resemble the density curve of $\mathrm{Beta}(1/2, 1/2)$, which is consistent with \eqref{eq:exmp1}.

\begin{figure}[htpb]
	\centering
	\includegraphics[scale=0.65]{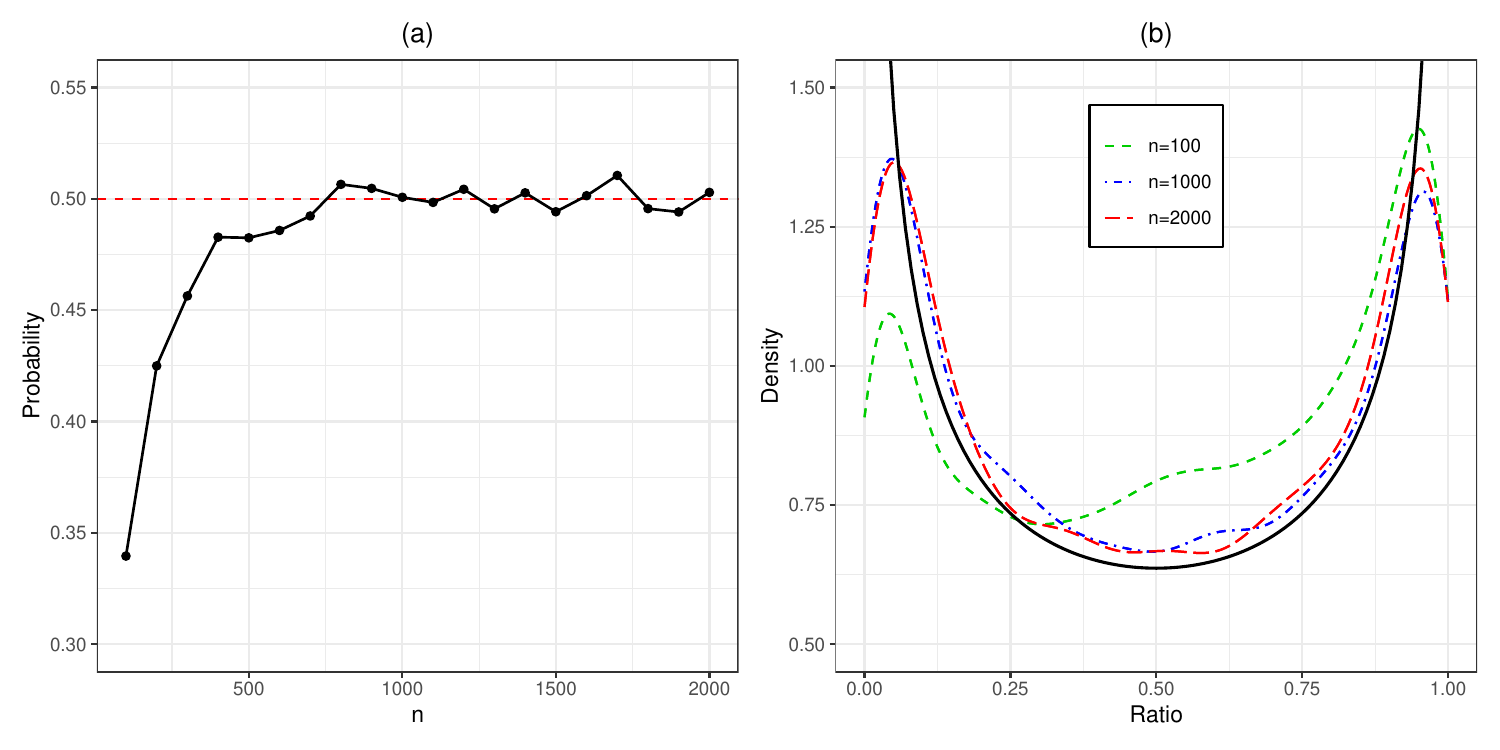}
	\caption{Simulation results for Example \ref{exam:toy}. (a): the simulated values of $\Pr(\mathbf{w}^L=\mathbf{w}_2^0)$. (b): the kernel density estimate for the values of $L_n(\mathbf{w}^L)/L_n(\mathbf{w}_2^0)$ after excluding 1, where the solid line is the density curve of $\mathrm{Beta}(1/2, 1/2)$.}\label{fig:exm1}
\end{figure}

In Lemma \ref{lem:loss} and Theorem \ref{thm:opt.loss} of the next section, we demonstrate that the above findings in Example \ref{exam:toy} remain consistent when the true model dimension is fixed. These results are opposite to the corresponding results obtained by \citet{fang2023} under the restricted weight set $\mathcal{H}_n^{\delta}$, which is somewhat surprising given our intuition. In the following section, we will also seek conditions that determine whether or not the asymptotic risk optimality \eqref{eq:risk} of $\hat{\mathbf{w}}$ holds, a topic that has not been addressed by \citet{fang2023}.

\section{Main Results}\label{sec:mainres}

In this section, we explore both the asymptotic loss and risk optimality of $\hat{\mathbf{w}}$. As pointed out in Subsection \ref{subsec:discu}, the result $\Pr(\mathbf{w}^L=\mathbf{w}_{M_0+1}^0)\to 1$ may not hold. This suggests that the proof procedures used by \citet{fang2023}, i.e., Steps 1 and 2 in Subsection \ref{subsec:review}, may be not applicable. Observe that
\begin{equation*}
	\frac{L_n(\hat{\mathbf{w}})}{\inf_{\mathbf{w}\in \mathcal{H}_n} L_n(\mathbf{w})}=\frac{L_n(\hat{\mathbf{w}})}{L_n(\mathbf{w}_{M_0+1}^0)}\times \frac{L_n(\mathbf{w}_{M_0+1}^0)}{\inf_{\mathbf{w}\in \mathcal{H}_n} L_n(\mathbf{w})}.
\end{equation*}
Therefore, if $L_n(\hat{\mathbf{w}})/L_n(\mathbf{w}_{M_0+1}^0)\to_p 1$, then the asymptotic loss optimality holds if and only if $L_n(\mathbf{w}_{M_0+1}^0)/\inf_{\mathbf{w}\in \mathcal{H}_n} L_n(\mathbf{w})$ conveges to $1$ in probability. Similarly, if $R_n(\hat{\mathbf{w}})/R_n(\mathbf{w}_{M_0+1}^0)\to_p 1$, then the asymptotic risk optimality holds if and only if $R_n(\mathbf{w}_{M_0+1}^0)/\inf_{\mathbf{w}\in \mathcal{H}_n} R_n(\mathbf{w})$ conveges to $1$. It is worth mentioning that in Appendixes \ref{subsec:thm1} and \ref{subsec:thm2}, we establish that $L_n(\hat{\mathbf{w}})/L_n(\mathbf{w}_{M_0+1}^0)\to_p 1$ and $R_n(\hat{\mathbf{w}})/R_n(\mathbf{w}_{M_0+1}^0)\to_p 1$ when $\phi_n=\log n$ under certain conditions. As a result, we can replace Step 1 in Subsection \ref{subsec:review} with the following:
\begin{enumerate}[\bf Step 1'.]
	\item Determine the conditions under which the true model is asymptotically loss or risk optimal, and the conditions under which it is not.
\end{enumerate}
Note that Step 1' does not depend on specific weight selection methods. In this section, we follow Step 1' and Step 2 to explore the asymptotic loss and risk optimality.

\subsection{Regularity Conditions}

Throughout this section, we assume that $E|e_i|^{4+\zeta}<\infty$ for some $\zeta>0$. Let $\lambda_{\min}(\mathbf{M})$ be the minimum eigenvalue of a positive definite matrix $\mathbf{M}$. We first state the regularity conditions required for main results.
\begin{condition}\label{cond:mu}
	$\|\bm{\mu}\|^2/n=O(1)$.
\end{condition}
\begin{condition}\label{cond:eigvalue}
	There exists a positive constant $\kappa_0$ such that $\lambda_{\min}(n^{-1}\mathbf{X}_{M_0+1}^{\top}\mathbf{X}_{M_0+1})\geq \kappa_0$.
\end{condition}
\begin{condition}\label{cond:sigma}
	There exist positive constants $0<c_1<c_2<\infty$ such that $\Pr(c_1\leq \hat{\sigma}^2/\sigma^2\leq c_2)\to 1$.
\end{condition}
\begin{condition}\label{cond:xcong}
	If $k_{M_0+1}$ is fixed, $n^{-1}\mathbf{X}_{M_0+1}^{\top}\mathbf{X}_{M_0+1}\to \mathbf{Q}$ and $n^{-1/2}\mathbf{X}_{M_0+1}^{\top} \mathbf{e}\to_d N(\mathbf{0}, \sigma^2\mathbf{Q})$, where $\mathbf{Q}$ is a positive definite matrix.
\end{condition}

Condition \ref{cond:mu} is a standard moment condition in the model averaging literature \citep{wan2010,liu2016}, which requires the average of $\mu_i^2$ is bounded. Condition \ref{cond:eigvalue} assumes that the eigenvalues of $n^{-1}\mathbf{X}_{M_0+1}^{\top}\mathbf{X}_{M_0+1}$ are uniformly
bounded away from zero for all $n$. This condition is commonly used for theoretical results with a diverging number of parameters \citep{zou2009,fang2023,zhang2020PMA}. Obviously, if the true model dimension is fixed and $\mathbf{X}_{M_0+1}$ is of full rank, Condition \ref{cond:eigvalue} is satisfied. Condition \ref{cond:sigma} is from \citet{yuan.yang2005}, \citet{zhang2020PMA}, and \citet{fang2023}, which essentially assumes that $\hat{\sigma}^2$ is neither too small nor too large. Note that Condition \ref{cond:sigma} does not require that $\hat{\sigma}^2$ is consistent and is thus easily satisfied. Condition \ref{cond:xcong} is adopted from \citet{liu2015joe} and \citet{zhang.liu2019}, which is quite mild. Note that this condition is not required if $k_{M_0+1}$ is diverging.

Let $\eta_n=\sum_{j\in \mathcal{A}_{M_0+1}\cap \mathcal{A}_{M_0}^c} \beta_j^2$. Note that $\eta_n$ measures the difference between the true model and the largest under-fitted model. Under Conditions \ref{cond:mu}--\ref{cond:eigvalue} and the assumption that the $(M_0+1)$th model is the true model, it can be easily verified that $0< \eta_n<\infty$. If the true model dimension is diverging, $\eta_n$ could converge to $0$. On the other hand, if the true model dimension is fixed, $\eta_n^{-1}=O(1)$.

\subsection{Asymptotic Loss Optimality}

In this subsection, we aim to determine the conditions under which $\hat{\mathbf{w}}$ achieves the asymptotic loss optimality \eqref{eq:loss}, and the conditions under which it does not in the sense that
\begin{equation*}
	\frac{L_n(\hat{\mathbf{w}})}{\inf_{\mathbf{w}\in \mathcal{H}_n} L_n(\mathbf{w})}\nrightarrow_p 1.
\end{equation*}

As mentioned at the beginning of this section, we first explore whether the true model is asymptotically loss optimal or not (Step 1').
\begin{lemma}\label{lem:loss}
	Assume that Conditions \ref{cond:mu}--\ref{cond:eigvalue} are satisfied.
	\begin{itemize}
	\item[(i)] If $k_{M_0+1}$ is fixed and Condition \ref{cond:xcong} is satisfied, then
	\begin{equation*}
		\frac{L_n(\mathbf{w}_{M_0+1}^0)}{\inf_{\mathbf{w}\in \mathcal{H}_n} L_n(\mathbf{w})} \nrightarrow_p 1.
	\end{equation*}
	\item[(ii)] If $k_{M_0+1}\to \infty$, $(k_{M_0+1}-k_1)/(n\eta_n)\to 0$, and $M_0/(k_{M_0+1}\eta_n)\to 0$, then
		\begin{equation*}
		\frac{L_n(\mathbf{w}_{M_0+1}^0)}{\inf_{\mathbf{w}\in \mathcal{H}_n} L_n(\mathbf{w})} \to_p 1.
	\end{equation*}
\end{itemize}
\end{lemma}

Lemma \ref{lem:loss}(i) indicates that when the true model dimension is fixed, the true model is not asymptotically loss optimal. This result also provides a negative response to Question Q1 in the case of a fixed true model dimension. In the proof of this result in Appendix \ref{subsec:lem1}, we actually derive a more general result. Specifically, we demonstrate that for any constant $z>1$,
\begin{equation}\label{eq:infp}
	\liminf_{n\to\infty} \Pr\left\{\frac{L_n(\mathbf{w}_{M_0+1}^0)}{\inf_{\mathbf{w}\in \mathcal{H}_n} L_n(\mathbf{w})}\geq z\right\}\geq \frac{1}{2} \Pr\left\{ \frac{(\mathbf{v}^{\top}\mathbf{Z})^2}{\|\mathbf{Z}\|^2}\geq 1-\frac{1}{z}\right\},
\end{equation}
where $\mathbf{Z}\sim N(\mathbf{0}, \mathbf{I}_{k_{M_0+1}})$ and $\mathbf{v}$ is a $k_{M_0+1}\times 1$ unit vector defined in \eqref{eq:unit} of the Appendix. Given that $\|\mathbf{v}\|=1$, it is easy to see that the support of $(\mathbf{v}^{\top}\mathbf{Z})^2/\|\mathbf{Z}\|^2$ is $[0, 1]$. By using Theorems 3--4 of \citet{forchini2002}, it is possible to derive an exact distribution function of $(\mathbf{v}^{\top}\mathbf{Z})^2/\|\mathbf{Z}\|^2$. In the special case where $\mathbf{Q}=\sigma_X^2 \mathbf{I}_{k_{M_0+1}}$ with some $\sigma_X^2>0$ and $k_{M_0+1}-k_{M_0}=1$, we obtain $\mathbf{v}=(0, \ldots, 0, \pm1)^{\top}$ and $(\mathbf{v}^{\top}\mathbf{Z})^2/\|\mathbf{Z}\|^2 \sim \mathrm{Beta}(1/2, k_{M_0}/2)$. Therefore, from \eqref{eq:infp}, we can establish Lemma \ref{lem:loss}(i).

In Lemma \ref{lem:loss}(ii), the condition $M_0/(k_{M_0+1}\eta_n)\to 0$ constrains the rate of increase in the number of under-fitted models $M_0$. Therefore, Lemma \ref{lem:loss}(ii) tells us that when the true model dimension diverges, but not too fast, and the number of under-fitted models does not diverge too fast, the true model is asymptotically loss optimal.

\begin{remark}\label{rem:special.w}
	We can also explore the asymptotic loss optimality of the true model when the weights are restricted to the special weight set $\mathcal{H}_n(N)$ or $\mathcal{H}_n^{\delta}$. Assuming that Conditions \ref{cond:mu}--\ref{cond:eigvalue} are satisfied, the specific results are as follows.
	\begin{itemize}
	\item[(i)] If $(k_{M_0+1}-k_1)/(n\eta_n)\to 0$ and $M_0/(n\eta_n^2)\to 0$, then $\Pr\{\inf_{\mathbf{w}\in \mathcal{H}_n(N)} L_n(\mathbf{w})=L_n(\mathbf{w}_{M_0+1}^0)\}\to 1$.
	\item[(ii)] Assume that $\eta_n\geq c_{\tau_0} n^{-\tau_0}$ for a positive constant $c_{\tau_0}$. If $(k_{M_0+1}-k_1)/n^{1-2\tau_0}\to 0$ and $M_0/n^{1-4\tau_0}\to 0$, then $\Pr\{\inf_{\mathbf{w}\in \mathcal{H}_n^{\delta}} L_n(\mathbf{w})=L_n(\mathbf{w}_{M_0+1}^0)\}\to 1$.
\end{itemize}
The proofs are provided in Appendix \ref{subsec:special.w}. In particular, if $k_{M_0+1}$ is fixed, then
\begin{equation*}
	\frac{L_n(\mathbf{w}_{M_0+1}^0)}{\inf_{\mathbf{w}\in \mathcal{H}_n(N)} L_n(\mathbf{w})} \to_p 1\quad \text{and}\quad \frac{L_n(\mathbf{w}_{M_0+1}^0)}{\inf_{\mathbf{w}\in \mathcal{H}_n^{\delta}} L_n(\mathbf{w})} \to_p 1,
\end{equation*}
which differ from the result in Lemma \ref{lem:loss}(i). This also indicates that model averaging with the weight sets $\mathcal{H}_n$ and $\mathcal{H}_n(N)$ (or $\mathcal{H}_n^{\delta}$) can exhibit opposite asymptotic properties.
\end{remark}

The following theorem establishes the asymptotic loss optimality of the least squares model averaging with $\phi_n=2$ and $\phi_n=\log n$.
\begin{theorem}\label{thm:opt.loss}
	Assume that Conditions \ref{cond:mu}--\ref{cond:xcong} are satisfied.
	\begin{enumerate}[(i)]
		\item Consider $\phi_n\to \infty$. If $k_{M_0+1}$ is fixed and $\phi_n^2/n \to 0$, then the asymptotic loss optimality does not hold. If $k_{M_0+1}\to \infty$, $\phi_n^2(k_{M_0+1}-k_1)/(n\eta_n)\to 0$, and $M_0/(k_{M_0+1}\eta_n)\to 0$, then the asymptotic loss optimality holds.
		\item Consider $\phi_n=2$. If $(k_{M_0+1}-k_1)/(n\eta_n)\to 0$ and $M_0/(n\eta_n^2)\to 0$, then as long as $\sum_{m>M_0+1} \hat{w}_m^2 (\frac{\mathbf{e}^{\top}\mathbf{P}_m\mathbf{e}}{\mathbf{e}^{\top}\mathbf{P}_{M_0+1}\mathbf{e}}-1)\nrightarrow_p 0$, the asymptotic loss optimality does not hold. If $k_{M_0+1}\to \infty$, $(k_{M_0+1}-k_1)/(n\eta_n)\to 0$, $M_0/(k_{M_0+1}\eta_n)\to 0$, and $k_{M_n}/k_{M_0+1}\to 1$, then the asymptotic loss optimality holds.
	\end{enumerate}
\end{theorem}

The first part of Theorem \ref{thm:opt.loss}(i) reveals that when $\phi_n=\log n$ and the true model dimension is fixed, the least squares model averaging is not asymptotically loss optimal because the asymptotic loss optimality of the true model does not hold. This finding is opposite to the corresponding result of \citet{fang2023} and thus partially provides a negative response to Question Q2. In the proof of this result in Appendix \ref{subsec:thm1}, we derive a more general result. Specifically, we demonstrate that if $\phi_n\to \infty$ and $\phi_n^2/n \to 0$, then for any constant $z>1$,
\begin{equation}\label{eq:infpw}
	\liminf_{n\to\infty} \Pr\left\{\frac{L_n(\hat{\mathbf{w}})}{\inf_{\mathbf{w}\in \mathcal{H}_n} L_n(\mathbf{w})}\geq z\right\}\geq \frac{1}{2} \Pr\left\{ \frac{(\mathbf{v}^{\top}\mathbf{Z})^2}{\|\mathbf{Z}\|^2}\geq 1-\frac{1}{z}\right\},
\end{equation}
where $\mathbf{Z}$ and $\mathbf{v}$ are the same as those in \eqref{eq:infp}. It follows from \eqref{eq:infpw} that the asymptotic loss optimality of $\hat{\mathbf{w}}$ does not hold.

The second part of Theorem \ref{thm:opt.loss}(i) suggests that when $\phi_n=\log n$, the true model dimension diverges, but not too fast, and the number of under-fitted models does not diverge too fast, the least squares model averaging is asymptotically loss optimal. This result arises because the asymptotic loss optimality of the true model holds, and the model averaging method pushes all the weights to the true model.

When $\phi_n=2$ and the model dimensions are fixed, \citet{fang2023} established that $\sum_{m>M_0+1} \hat{w}_m^2 (\frac{\mathbf{e}^{\top}\mathbf{P}_m\mathbf{e}}{\mathbf{e}^{\top}\mathbf{P}_{M_0+1}\mathbf{e}}-1)$ does not converge to $0$ in probability. This is because $\frac{\mathbf{e}^{\top}\mathbf{P}_m\mathbf{e}}{\mathbf{e}^{\top}\mathbf{P}_{M_0+1}\mathbf{e}}\nrightarrow_p 1$, and $\{\hat{w}_m, m\geq M_0+1\}$ converges to a nondegenerate random weight vector in distribution \citep{zhang.liu2019}. So the first part of Theorem \ref{thm:opt.loss}(ii) actually shows that MMA is not asymptotically loss optimal with fixed model dimensions because it assigns nonnegligible weights to over-fitted models. This result coincides with the corresponding result of \citet{fang2023}. It is important to note that another reason for this result is that the asymptotic loss optimality of the true model does not hold.

The second part of Theorem \ref{thm:opt.loss}(ii) implies that when the true model dimension diverges, MMA is asymptotically loss optimal under two key conditons. The first condition is that $(k_{M_0+1}-k_1)/(n\eta_n)\to 0$ and $M_0/(k_{M_0+1}\eta_n)\to 0$, ensuring the asymptotic loss optimality of the true model. The other condition is that $k_{M_n}/k_{M_0+1}\to 1$, indicating that the over-fitted models have no significant difference from the true model in the large sample sense. Although MMA may assign nonnegligible weights to over-fitted models, the condition $k_{M_n}/k_{M_0+1}\to 1$ ensures that these over-fitted models can be considered negligible.

\begin{remark}\label{rem:subgauss}
	If $e_i$ is further assumed to be sub-Gaussian with a variance proxy $\bar{\sigma}^2$, that is $E\{\exp(\alpha e_i)\}\leq \exp(\bar{\sigma}^2 \alpha^2/2)$ for all $\alpha\in \mathbb{R}$, then the conditions $M_0/(k_{M_0+1}\eta_n)\to 0$ and $M_0/(n\eta_n^2)\to 0$ in Lemma \ref{lem:loss} and Theorem \ref{thm:opt.loss} can be relaxed to $\log(2M_0)/(k_{M_0+1}\eta_n)\to 0$ and $\log(2M_0)/(n\eta_n^2)\to 0$, respectively. See Appendix \ref{subsec:subg} for more details.
\end{remark}

\subsection{Asymptotic Risk Optimality}

In this subsection, we aim to determine the conditions under which $\hat{\mathbf{w}}$ achieves the asymptotic risk optimality \eqref{eq:risk}, and the conditions under which it does not in the sense that
\begin{equation*}
	\frac{R_n(\hat{\mathbf{w}})}{\inf_{\mathbf{w}\in \mathcal{H}_n} R_n(\mathbf{w})}\nrightarrow_p 1.
\end{equation*}

As mentioned at the beginning of this section, we first explore whether the true model is asymptotically risk optimal or not (Step 1').
\begin{lemma}\label{lem:risk}
	Assume that Conditions \ref{cond:mu}--\ref{cond:eigvalue} are satisfied. If $(k_{M_0+1}-k_1)/(n\eta_n)\to 0$, then
	\begin{equation*}
		\frac{R_n(\mathbf{w}_{M_0+1}^0)}{\inf_{\mathbf{w}\in \mathcal{H}_n} R_n(\mathbf{w})}\to 1.
	\end{equation*}
\end{lemma}

The condition $(k_{M_0+1}-k_1)/(n\eta_n)\to 0$ is clearly satisfied when the true model dimension is fixed. Therefore, Lemma \ref{lem:risk} implies that as long as the true model dimension does not diverge too fast, the true model is asymptotically risk optimal. Comparing this result with Lemma \ref{lem:loss}(i), an intriguing observation is that when the true model dimension is fixed, the true model is not asymptotically loss optimal, but it is asymptotically risk optimal.

Now we explore the asymptotic risk optimality of the least squares model averaging with $\phi_n=2$ and $\phi_n=\log n$ in the following theorem.
\begin{theorem}\label{thm:opt.risk}
	Assume that Conditions \ref{cond:mu}--\ref{cond:sigma} are satisfied.
	\begin{enumerate}[(i)]
		\item Consider $\phi_n\to \infty$. If $\phi_n^2 (k_{M_0+1}-k_1)/(n\eta_n^2)\to 0$, then the asymptotic risk optimality holds.
		\item Consider $\phi_n=2$. If $(k_{M_0+1}-k_1)/(n\eta_n^2)\to 0$, then as long as $\sum_{m>M_0+1} \hat{w}_m^2 (\frac{k_m}{k_{M_0+1}}-1)\nrightarrow_p 0$, the asymptotic risk optimality does not hold. If $k_{M_0+1}\to \infty$, $(k_{M_0+1}-k_1)/(n\eta_n^2)\to 0$, and $k_{M_n}/k_{M_0+1}\to 1$, then the asymptotic risk optimality holds.
	\end{enumerate}
\end{theorem}

Theorem \ref{thm:opt.risk}(i) indicates that when $\phi_n=\log n$, as long as the true model dimension does not diverge too fast, the least squares model averaging is asymptotically risk optimal. This is because the true model achieves the asymptotic risk optimality, and the model averaging method pushes all the weights to the true model. Comparing this result with Theorem \ref{thm:opt.loss}(i), a remarkable observation emerges: when $\phi_n=\log n$ and the true model dimension is fixed, the least squares model averaging is not asymptotically loss optimal, but it is asymptotically risk optimal.

When $\phi_n=2$ and the model dimensions are fixed, Theorem 2 of \citet{zhang.liu2019} established that $\{\hat{w}_m, m\geq M_0+1\}$ converges in distribution to a nondegenerate random weight vector. Combining this result with the fact that $\sum_{m>M_0+1} \hat{w}_m^2 (\frac{k_m}{k_{M_0+1}}-1)\geq (\frac{k_{M_0+2}}{k_{M_0+1}}-1) \sum_{m>M_0+1} \hat{w}_m^2$, it follows that $\sum_{m>M_0+1} \hat{w}_m^2 (\frac{k_m}{k_{M_0+1}}-1)$ does not converge to $0$ in probability. Therefore, the first part of Theorem \ref{thm:opt.risk}(ii) actually shows that MMA is not asymptotically risk optimal with fixed model dimensions since it puts weights on the over-fitted models with a positive probability.

The second part of Theorem \ref{thm:opt.risk} implies that when the true model dimension diverges, but not too fast, and $k_{M_n}/k_{M_0+1}\to 1$, MMA is asymptotically risk optimal. The condtion $k_{M_n}/k_{M_0+1}\to 1$ means that the over-fitted models do not significantly differ from the true model in the large sample sense. Therefore, in this case, MMA's asymptotic risk optimality holds because these over-fitted models can be considered negligible, even though MMA might assign non-negligible weights to them.

Last, we verify Lemma \ref{lem:risk} and Theorem \ref{thm:opt.risk} in the following examples.
\begin{continueexample}{exam:toy}
	Using Lemma \ref{lem:lossrisk} in the Appendix, we can rewrite $R_n(\mathbf{w})$ as $R_n(\mathbf{w})=R_n(\mathbf{w}_2^0)+w_1^2 (n\beta_2^2+\sigma^2)-2w_1 \sigma^2+w_3^2 \sigma^2$. Let $\mathbf{w}^R=(w_1^R,w_2^R,w_3^R)^{\top}=\argmin_{\mathbf{w}\in \mathcal{H}_n} R_n(\mathbf{w})$. It is easy to obtain $w_1^R=\sigma^2/(n\beta_2^2+\sigma^2)$, $w_2^R=1-w_1^R$, and $w_3^R=0$. Combining this with the fact that $R_n(\mathbf{w}_2^0)=2\sigma^2$, we have
	\begin{equation}\label{eq:exmp2}
		\frac{\inf_{\mathbf{w}\in \mathcal{H}_n} R_n(\mathbf{w})}{R_n(\mathbf{w}_2^0)}=\frac{R_n(\mathbf{w}^R)}{R_n(\mathbf{w}_2^0)}=1-\frac{\sigma^2}{2(n\beta_2^2+\sigma^2)}=1+O\left(\frac{1}{n}\right).
	\end{equation}
	This verifies the assertion in Lemma \ref{lem:risk}. Additionally, Lemma 2 of \citet{zhang2020PMA} established that if $\phi_n/n\to 0$, then $\hat{w}_1=O_p(\phi_n/n)$. This leads to
	\begin{equation}\label{eq:exmp3}
		\frac{R_n(\hat{\mathbf{w}})}{R_n(\mathbf{w}_2^0)}=1+O_p\left(\frac{\phi_n^2}{n}\right)+\frac{1}{2}\hat{w}_3^2.
	\end{equation}
	When $\phi_n=\log n$, Lemma 1 of \citet{zhang2020PMA} implied that $\Pr(\hat{w}_3=0)\to 1$. In the case of $\phi_n=2$, Theorem 2 of \citet{zhang.liu2019} established that $(\hat{w}_2, \hat{w}_3)^{\top}\to_d (\tilde{\lambda}_1, \tilde{\lambda}_2)^{\top}=\argmin_{\bm{\lambda}\in \mathcal{L}} \bm{\lambda}^{\top} \bm{\Lambda}\bm{\lambda}$, where $\mathcal{L}=\{\bm{\lambda}=(\lambda_1, \lambda_2)^{\top}\in [0, 1]^2\colon \lambda_1+\lambda_2=1\}$ and
	\begin{equation*}
		\bm{\Lambda}=\sigma^2\begin{pmatrix}
			4-\sum_{j=1}^2 Z_i^2 & 5-\sum_{j=1}^3 Z_i^2 \\
			5-\sum_{j=1}^3 Z_i^2 & 6-\sum_{j=1}^3 Z_i^2
		\end{pmatrix}.
	\end{equation*}
	Here, $Z_1, Z_2, Z_3\sim N(0, 1)$ are independent. A direct calculation yields $\tilde{\lambda}_2=\max\{1-1/Z_3^2, 0\}$. Combining these results with \eqref{eq:exmp2} and \eqref{eq:exmp3}, we conclude
	\begin{equation*}
		\frac{R_n(\hat{\mathbf{w}})}{\inf_{\mathbf{w}\in \mathcal{H}_n} R_n(\mathbf{w})}\begin{cases}
			=1+O_p(n^{-1} \log^2 n), & \mathrm{if}~\phi_n=\log n, \\
			\to_d 1+V^2/2, & \mathrm{if}~\phi_n=2,
		\end{cases}
	\end{equation*}
	where $V=\max\{1-1/\chi^2(1), 0\}$ and $\chi^2(1)$ is the chi-squared distribution with 1 degree of freedom. This confirms partial results of Theorem \ref{thm:opt.risk}.
\end{continueexample}

\subsection{Summary of the Results}\label{subsec:summ}

We summarize the main theoretical results in Table \ref{tab:summ}. One remarkable result is that when $\phi_n=\log n$ and the true model dimension is fixed, the asymptotic loss optimality does not hold, while the asymptotic risk optimality does. This occurs because when the true model dimension is fixed, the true model is not asymptotically loss optimal but is asymptotically risk optimal. This result does not contradict the corresponding finding of \citet{fang2023}, where they demonstrated the asymptotic loss optimality with a fixed true model dimension when $\phi_n=\log n$, since their result is confined to the special weight set $\mathcal{H}_n^{\delta}$.
\begin{table}[htpb]
	\begin{threeparttable}
		\centering
		\caption{Summary of the asymptotic loss and risk optimality for least squares model averaging when the true model is among the candidate models.}\label{tab:summ}
		\begin{tabular}{ccc}
			\hline
			$\phi_n$ & Asymptotic loss optimality & Asymptotic risk optimality \\
			\hline
			\multirow{2}{*}{$2$} & No (fixed model dimensions) & No (fixed model dimensions) \\
			& Yes (diverging case$^\dagger$, $k_{M_n}/k_{M_0+1}\to 1$) & Yes (diverging case$^\ddagger$, $k_{M_n}/k_{M_0+1}\to 1$) \\
			\hline
			\multirow{2}{*}{$\log n$} & No (fixed true model dimension)  & Yes (fixed true model dimension) \\
			& Yes (diverging case$^\dagger$) & Yes (diverging case$^\ddagger$) \\
			\hline
		\end{tabular}
		\begin{tablenotes}
			\small
			\item $^\dagger$: $k_{M_0+1}\to \infty$, $\phi_n^2(k_{M_0+1}-k_1)/(n\eta_n)\to 0$, and $M_0/(k_{M_0+1}\eta_n)\to 0$.
			\item $^\ddagger$: $k_{M_0+1}\to \infty$ and $\phi_n^2 (k_{M_0+1}-k_1)/(n\eta_n^2)\to 0$.
		\end{tablenotes}
	\end{threeparttable}
\end{table}

From Table \ref{tab:summ}, we can observe three cases where the asymptotic loss or risk optimality does not hold. These can be attributed to one or both of the following main reasons:
\begin{itemize}
	\item[(R1)] The true model is not asymptotically loss optimal.
	\item[(R2)] Least squares model averaging assigns weights to the over-fitted models with a positive probability.
\end{itemize}
Specifically, when model dimensions are fixed, MMA is not asymptotically loss optimal due to both Reasons (R1) and (R2), and it is also not asymptotically risk optimal because of Reason (R2). When $\phi_n=\log n$ and the true model dimension is fixed, the asymptotic loss optimality does not hold due to Reason (R1).

When $\phi_n=\log n$ and the true model dimension diverges, the least squares model averaging is asymptotically both loss and risk optimal under certain conditons on $M_0$, $k_m$, and $\eta_n$. This can be attributed to the asymptotic loss and risk optimality of the true model, along with the model averaging method assigning all the weights to the true model.

In the case where the true model dimension diverges, MMA is asymptotically both loss and risk optimal under certain conditons. This result is mainly attributed to the condition that $k_{M_n}/k_{M_0+1}\to 1$, rendering the over-fitted models asymptotically negligible, although MMA may assign weights to them with a positive probability.

Last, as mentioned in Remark \ref{rem:scenarios1}, the results from \citet{wan2010}, \citet{zhang2020PMA}, \citet{zhang2021}, and \citet{peng.opt} also indicate asymptotic optimality in cases where the true model dimension diverges. Here, we compare their conditions with ours, using the recent work by \citet{peng.opt} as an example. In the following discussion, we assume that $e_i$ is sub-Gaussian and Conditions \ref{cond:mu}--\ref{cond:eigvalue} are satisfied. Define
\begin{equation*}
	\psi(\mathbf{K})=\left[M_n \wedge \left\{1+\sum_{m=1}^{M_n-1} \frac{k_{m+1}-k_m}{4 k_m}+\sum_{m=1}^{M_0} \frac{\bm{\mu}^{\top} (\mathbf{P}_m-\mathbf{P}_{m-1})\bm{\mu}}{4 \bm{\mu}^{\top} (\mathbf{P}_{M_0+1}-\mathbf{P}_m)\bm{\mu}}\right\}\right] (1+\log M_n)^2,
\end{equation*}
where $\mathbf{P}_0=\mathbf{0}$ and $a\wedge b=\min(a, b)$ for any $a, b\in \mathbb{R}$. Under the conditions that $k_{M_0+1}\to \infty$ and $(k_{M_0+1}-k_1)/(n\eta_n)\to 0$, we can deduce from Lemma \ref{lem:risk} that $\xi_n/(\sigma^2 k_{M_0+1})\to 1$. Combining this with Corollary 1 of \citet{peng.opt}, we conclude that if $k_{M_0+1}\to \infty$, $(k_{M_0+1}-k_1)/(n\eta_n)\to 0$, $k_{M_n} E\{(\hat{\sigma}^2-\sigma^2)^2\}/k_{M_0+1}\to 0$, and $\psi(\mathbf{K})/k_{M_0+1}\to 0$, then MMA is asymptotically both loss and risk optimal. In comparison with the second part of Theorem \ref{thm:opt.loss}(ii) and the second part of Theorem \ref{thm:opt.risk}(ii), their condition $\psi(\mathbf{K})/k_{M_0+1}\to 0$ imposes a weaker restriction for the over-fitted models than our condition $k_{M_n}/k_{M_0+1}\to 1$. However, their condition $k_{M_n} E\{(\hat{\sigma}^2-\sigma^2)^2\}/k_{M_0+1}\to 0$ requires the consistency of $\hat{\sigma}^2$, which is not necessary for our results from Condition \ref{cond:sigma}.

\section{Simulation Studies}\label{sec:simu}

In this section, we conduct several simulation studies to illustrate the theoretical results presented in Section \ref{sec:mainres}. For each $i=1,\ldots,n$, we generate $y_i$ from the model \eqref{eq:model}, where the covariates $(x_{i1}, x_{i2}, \ldots)$ are i.i.d. normal random vectors with zero mean and the covariance matrix between the $(k, l)$th element being $\rho^{|k-l|}$, and the random errors $e_i$ are i.i.d. from $N(0, \sigma^2)$ and are independent of $x_{ij}$'s. The population R-squared is denoted as $R^2=\mathrm{Var}(\mu_i)/\mathrm{Var}(y_i)$, which is controlled in the range of $\{0.05, 0.1, 0.5, 0.9\}$ via the parameter $\sigma^2$. For each configuration, the simulation is repeated 10000 times.

We consider least squares model averaging methods with $\phi_n=2$ and $\phi_n=\log n$. To assess the results in Lemma \ref{lem:loss} and Theorem \ref{thm:opt.loss}, we calculate the averages of $L_n(\mathbf{w}_{M_0+1}^0)/\inf_{\mathbf{w}\in \mathcal{H}_n} L_n(\mathbf{w})$ and $L_n(\hat{\mathbf{w}})/\inf_{\mathbf{w}\in \mathcal{H}_n} L_n(\mathbf{w})$. Similarly, to evaluate the results in Lemma \ref{lem:risk} and Theorem \ref{thm:opt.risk}, we calculate the averages of $R_n(\mathbf{w}_{M_0+1}^0)/\inf_{\mathbf{w}\in \mathcal{H}_n} R_n(\mathbf{w})$ and $R_n(\hat{\mathbf{w}})/\inf_{\mathbf{w}\in \mathcal{H}_n} R_n(\mathbf{w})$. We illustrate the performance through the following three examples, each involving different settings for $\beta_j$.

\begin{exam}[Fixed true model dimension]\label{exam:fixdim}
	Set $\rho=0.5$, $(\beta_1,\ldots,\beta_5)=(1, -2, 3, 1.5, 4)$, and $\beta_j=0$ for $j\geq 6$. We consider $\mathbf{K}=(1,\ldots,11)^{\top}$. Then, $M_n=11$, $M_0=4$, $k_{M_0+1}=5$, and there are six over-fitted models. The sample size $n$ varies with 100, 1000, 10000, and 50000. Table \ref{tab:simu.exm1} presents the loss and risk ratios, which are summarized as follows.
	\begin{itemize}
		\item[(i)] For each $R^2$, the risk ratio for the true model converges to 1, while the loss ratio does not. This indicates that the asymptotic loss optimality of the true model does not hold, but its asymptotic risk optimality does, confirming Lemmas \ref{lem:loss}(i) and \ref{lem:risk}.
		\item[(ii)] For each $R^2$, the loss ratio for $\hat{\mathbf{w}}$ with $\phi_n=\log n$ is larger than 1.5 even for large sample sizes, and the risk ratio converges to 1. This means that the least squares model averaging with $\phi_n=\log n$ is not asymptotically loss optimal but is asymptotically risk optimal, confirming the first part of Theorem \ref{thm:opt.loss}(i) and Theorem \ref{thm:opt.risk}(i).
		\item[(iii)] For each $R^2$, the loss ratio for $\hat{\mathbf{w}}$ with $\phi_n=2$ is larger than 2 even for large sample sizes, and the risk ratio stabilizes near $1.08$. This implies that MMA is not neither asymptotically loss optimal nor asymptotically risk optimal, coinciding with the first part of Theorem \ref{thm:opt.loss}(ii) and the first part of Theorem \ref{thm:opt.risk}(ii).
	\end{itemize}
	
	Next, we verify the two inequalities \eqref{eq:infp} and \eqref{eq:infpw} using the same simulation settings as before, except for setting $\rho=0$. In this case, as mentioned below \eqref{eq:infp}, we have $\mathbf{Q}=\mathbf{I}_{k_{M_0+1}}$ and $(\mathbf{v}^{\top}\mathbf{Z})^2/\|\mathbf{Z}\|^2 \sim \mathrm{Beta}(1/2, 2)$. The upper panel of Figure \ref{fig:prob.ineq} illustrates the simulated value of $\Pr\{L_n(\mathbf{w}_{M_0+1}^0)/L_n(\mathbf{w}^L)\geq z\}$ from $z=1.01$ to 5 with an increment of 0.01, where the solid line is the function $2^{-1}\Pr\{\mathrm{Beta}(1/2, 2)\geq z\}$. For each $R^2$, it is clear that the solid line is consistently the lowest among the curves of the simulated value for different sample sizes, confirming \eqref{eq:infp}. Similarly, the lower panel of Figure \ref{fig:prob.ineq} displays the simulated value of $\Pr\{L_n(\hat{\mathbf{w}})/L_n(\mathbf{w}^L)\geq z\}$ for $z\in (1, 5]$, which verifies \eqref{eq:infpw}.
\end{exam}

\begin{exam}[Diverging true model dimension and $M_0/(k_{M_0+1}\eta_n)\to 0$]\label{exam:divdim1}
	Set $\rho=0.5$, $\beta_j=j^{-1}$ for $j=1,\ldots, p_n-1$, $\beta_{p_n}=1$, and $\beta_j=0$ for $j>p_n$, where $p_n=\lfloor 2n^{1/3}\rfloor$ and $\lfloor \cdot \rfloor$ returns the floor of $\cdot$. We consider $\mathbf{K}=(p_n-3,\ldots,p_n+5)^{\top}$. Then, $M_n=9$, $M_0=3$, $k_{M_0+1}=p_n$, and there are five over-fitted models. The sample size $n$ varies with 100, 1000, 10000, 50000, and 100000. It is easy to see that $\eta_n=1$, $M_0/(k_{M_0+1}\eta_n)\to 0$, $k_{M_n}/k_{M_0+1}\to 1$, and the remaining conditions of our results are satisfied. Table \ref{tab:simu.exm2} presents the loss and risk ratios. Our observations are as follows. First, the loss and risk ratios for the true model converge to 1, confirming Lemmas \ref{lem:loss}(ii) and \ref{lem:risk}. Second, the loss and risk ratios for MMA converge to 1. This indicates MMA's asymptotic loss and risk optimality for this example, verifying the second part of Theorem \ref{thm:opt.loss}(ii) and the second part of Theorem \ref{thm:opt.risk}(ii). Third, the loss and risk ratios for $\hat{\mathbf{w}}$ with $\phi_n=\log n$ converge to 1, which is consistent with the second part of Theorem \ref{thm:opt.loss}(i) and Theorem \ref{thm:opt.risk}(i).
\end{exam}

\begin{exam}[Diverging true model dimension and $M_0/(k_{M_0+1}\eta_n)\nrightarrow 0$]\label{exam:divdim2}
	Set $\rho=0.5$, $\beta_j=j^{-1}$ for $j=1,\ldots, p_n-1$, $\beta_{p_n}=5/\log\log n$, and $\beta_j=0$ for $j>p_n$, where $p_n=\lfloor \log n\rfloor$. We consider $\mathbf{K}=(1,\ldots,p_n+3)^{\top}$. Then, $M_n=p_n+3$, $M_0=p_n-1$, $k_{M_0+1}=p_n$, and there are three over-fitted models. The sample size $n$ varies with 100, 1000, 10000, 100000, and 500000. It can be easily verified that $\eta_n=(5/\log\log n)^2\to 0$, $k_{M_n}/k_{M_0+1}\to 1$, and the remaining conditions in Lemma \ref{lem:risk} and Theorem \ref{thm:opt.risk} are satisfied. The simulation results for the loss and risk ratios are presented in Table \ref{tab:simu.exm3}. First, the risk ratios for the true model and $\hat{\mathbf{w}}$ with $\phi_n=2$ and $\phi_n=\log n$ converge to 1, confirming Lemma \ref{lem:risk}, Theorem \ref{thm:opt.risk}(i), and the second part of Theorem \ref{thm:opt.risk}(ii). It is worth noting that the risk ratio for MMA appears to converge slowly to 1, which can be attributed to the very slow convergence rate of $k_{M_n}/k_{M_0+1}$. Second, it seems that the loss ratios for the true model and $\hat{\mathbf{w}}$ with $\phi_n=2$ and $\phi_n=\log n$ do not converge to 1, but this observation cannot be explained by the results in this paper. The reason is that the condition $M_0/(k_{M_0+1}\eta_n)\to 0$ in Lemma \ref{lem:loss}(ii) and Theorem \ref{thm:opt.loss} is not satisfied, and thus it is unknown whether the asymptotic loss optimality holds or not for this example.
\end{exam}

\section{Concluding Remarks}\label{sec:con}

In this article, we study the asymptotic optimality of least squares model averaging with the weight set $\mathcal{H}_n$ for nested candidate models in the scenario where the true model is included in the set of candidate models. In contrast to the existing work by \citet{fang2023}, our work is not confined to their special weight set $\mathcal{H}_n^{\delta}$, and we also explore the asymptotic risk optimality. A surprising finding is that when $\phi_n=\log n$ and the true model dimension is fixed, the asymptotic loss optimality does not hold. This result differs from the corresponding result of \citet{fang2023} and the asymptotic behavior of BIC. The reason behind this discrepancy is that the asymptotic loss optimality of the true model does not hold, while it holds for the restricted weight sets $\mathcal{H}_n(N)$ and $\mathcal{H}_n^{\delta}$.

In conclusion, we present several interesting directions for future research. First, the results in Lemma \ref{lem:loss}(i) and the first part of Theorem \ref{thm:opt.loss}(i) were derived under the assumption of a fixed true model dimension. It remains an open question whether these results still hold in the case of a diverging true model dimension. Therefore, an appealing direction to relax the conditions of the results in this paper. Second, it is of interest to investigate the asymptotic optimality of jackknife model averaging \citep{hansen2012} in the scenario of this paper under heteroskedastic errors. Third, our current results are confined to nested models. It is desirable to extend the analysis to include non-nested models, but it is quite challenging as previously mentioned by \cite{hansen2014}, \citet{zhang.liu2019}, and \citet{fang2023}. Last, it would be intriguing to explore similar issues in the context of other models, such as generalized linear models \citep{zhang2016jasa,yu.glm2024}, quantile regression models \citep{lu2015}, and time series models \citep{liao2021.joe}. These problems warrant further investigations.

\renewcommand{\theequation}{{A.\arabic{equation}}}
\renewcommand{\thelemma}{{\it A.\arabic{lemma}}}
\renewcommand{\thesubsection}{{\it A.\arabic{subsection}}}
\setcounter{equation}{0}
\setcounter{lemma}{0}

\section*{Appendix}

\subsection{Technical Lemmas}

\begin{lemma}\label{lem:lossrisk}
	For any $\mathbf{w}$ satisfying $\sum_{m=1}^{M_n} w_m=1$ (including both random and non-random $\mathbf{w}$’s), we have
	\begin{align*}
		L_n(\mathbf{w})&=L_n(\mathbf{w}_{M_0+1}^0)+\sum_{m=1}^{M_0} \left(w_m^2+2w_m\sum_{l=1}^{m-1} w_l\right) \mathbf{y}^{\top} (\mathbf{P}_{M_0+1}-\mathbf{P}_m)\mathbf{y} \\
		&\quad -2\sum_{m=1}^{M_0} w_m \mathbf{y}^{\top} (\mathbf{P}_{M_0+1}-\mathbf{P}_m)\mathbf{e}+\sum_{m>M_0+1} \left(w_m^2+2w_m \sum_{l=m+1}^{M_n} w_l\right) \mathbf{e}^{\top}(\mathbf{P}_m-\mathbf{P}_{M_0+1})\mathbf{e}
	\end{align*}
	and
	\begin{align*}
		R_n(\mathbf{w})&=R_n(\mathbf{w}_{M_0+1}^0)+\sum_{m=1}^{M_0} \left(w_m^2+2w_m\sum_{l=1}^{m-1} w_l\right) \left\{\bm{\mu}^{\top} (\mathbf{P}_{M_0+1}-\mathbf{P}_m)\bm{\mu}+\sigma^2 (k_{M_0+1}-k_m)\right\} \\
		&\quad -2\sigma^2 \sum_{m=1}^{M_0} w_m (k_{M_0+1}-k_m)+\sigma^2\sum_{m>M_0+1} \left(w_m^2+2w_m \sum_{l=m+1}^{M_n} w_l\right) (k_m-k_{M_0+1}).
	\end{align*}
\end{lemma}
\begin{proof}
	For any $\mathbf{w}$ satisfying $\sum_{m=1}^{M_n} w_m=1$, we have
	\begin{align*}
		L_n(\mathbf{w})&=\left\|\bm{\mu}-\sum_{m=1}^{M_n} w_m \mathbf{P}_m\mathbf{y}\right\|^2 \\
		&=\left\|\bm{\mu}-\sum_{m=1}^{M_0} w_m \mathbf{P}_m\mathbf{y}-w_{M_0+1}\mathbf{P}_{M_0+1}\mathbf{y}-\sum_{m>M_0+1} w_m \mathbf{P}_m\mathbf{y}\right\|^2 \\
		&=\left\|\bm{\mu}-\mathbf{P}_{M_0+1}\mathbf{y}+\sum_{m=1}^{M_0} w_m (\mathbf{P}_{M_0+1}-\mathbf{P}_m)\mathbf{y}-\sum_{m>M_0+1} w_m (\mathbf{P}_m-\mathbf{P}_{M_0+1})\mathbf{y}\right\|^2 \\
		&=\left\|-\mathbf{P}_{M_0+1}\mathbf{e}+\sum_{m=1}^{M_0} w_m (\mathbf{P}_{M_0+1}-\mathbf{P}_m)\mathbf{y}-\sum_{m>M_0+1} w_m (\mathbf{P}_m-\mathbf{P}_{M_0+1})\mathbf{e}\right\|^2,
	\end{align*}
	where the third equality is derived from $w_{M_0+1}=1-\sum_{m=1}^{M_0} w_m-\sum_{m>M_0+1} w_m$ and the last equality is due to the fact $\mathbf{P}_m \bm{\mu}=\bm{\mu}$ for $m> M_0$. Using the fact $\mathbf{P}_m\mathbf{P}_l=\mathbf{P}_{\min\{m, l\}}$ in the nested model setting, it can be easily verified that
	\begin{equation}\label{eq:basic}
		\left.
	\begin{aligned}
		&(\mathbf{P}_{M_0+1}-\mathbf{P}_m)(\mathbf{P}_{M_0+1}-\mathbf{P}_l)=\mathbf{P}_{M_0+1}-\mathbf{P}_{\max\{m, l\}},& &\text{if}~1\leq m, l\leq M_0, \\
		&(\mathbf{P}_m-\mathbf{P}_{M_0+1})(\mathbf{P}_l-\mathbf{P}_{M_0+1})=\mathbf{P}_{\min\{m,l\}}-\mathbf{P}_{M_0+1},& &\text{if}~m, l>M_0, \\
		&(\mathbf{P}_{M_0+1}-\mathbf{P}_m)(\mathbf{P}_l-\mathbf{P}_{M_0+1})=\mathbf{0},& &\text{if}~m\leq M_0<l.
	\end{aligned}\right\}
\end{equation}
	Therefore, we have
	\begin{align*}
		L_n(\mathbf{w})&=\mathbf{e}^{\top}\mathbf{P}_{M_0+1}\mathbf{e}+\left\|\sum_{m=1}^{M_0} w_m (\mathbf{P}_{M_0+1}-\mathbf{P}_m)\mathbf{y}\right\|^2-2\sum_{m=1}^{M_0} w_m \mathbf{y}^{\top}(\mathbf{P}_{M_0+1}-\mathbf{P}_m)\mathbf{e} \\
		&\quad +\left\|\sum_{m>M_0+1} w_m (\mathbf{P}_m-\mathbf{P}_{M_0+1})\mathbf{e}\right\|^2 \\
		&=L_n(\mathbf{w}_{M_0+1}^0)+\sum_{m=1}^{M_0}\sum_{l=1}^{M_0} w_mw_l \mathbf{y}^{\top}(\mathbf{P}_{M_0+1}-\mathbf{P}_{\max\{m,l\}})\mathbf{y}-2\sum_{m=1}^{M_0} w_m \mathbf{y}^{\top}(\mathbf{P}_{M_0+1}-\mathbf{P}_m)\mathbf{e} \\
		&\quad +\sum_{m>M_0+1}\sum_{l>M_0+1} w_mw_l \mathbf{e}^{\top}(\mathbf{P}_{\min\{m,l\}}-\mathbf{P}_{M_0+1})\mathbf{e},
	\end{align*}
	where the last equality is derived by $L_n(\mathbf{w}_{M_0+1}^0)=\mathbf{e}^{\top}\mathbf{P}_{M_0+1}\mathbf{e}$ and \eqref{eq:basic}. This completes the proof of the first part of Lemma \ref{lem:lossrisk}. Taking expectation over $\mathbf{e}$ of the both sides of the expression of $L_n(\mathbf{w})$, we can obtain the second part of Lemma \ref{lem:lossrisk}.
\end{proof}

\begin{lemma}\label{lem:mumu}
	If Condition \ref{cond:eigvalue} holds, then $\bm{\mu}^{\top} (\mathbf{P}_{M_0+1}-\mathbf{P}_m)\bm{\mu}\geq \kappa_0 n\sum_{j\in \mathcal{A}_{M_0+1}\cap \mathcal{A}_m^c} \beta_j^2$ for any $m\in \{1,\ldots,M_0\}$.
\end{lemma}
\begin{proof}
	Let $\mathbf{X}_m^c$ be the $n\times (k_{M_0+1}-k_m)$ design matrix that consists of the covariates $x_{ij}, j\in \mathcal{A}_{M_0+1}\cap \mathcal{A}_m^c$. Let $\bm{\beta}_m=(\beta_1,\ldots,\beta_{k_m})^{\top}$ and $\bm{\beta}_m^c=(\beta_{k_m+1},\ldots, \beta_{k_{M_0+1}})^{\top}$. Then $\mathbf{X}_{M_0+1}=(\mathbf{X}_m, \mathbf{X}_m^c)$, $\bm{\beta}_{M_0+1}=(\bm{\beta}_m^{\top}, \bm{\beta}_m^{c\top})^{\top}$, and $\bm{\mu}=\mathbf{X}_{M_0+1} \bm{\beta}_{M_0+1}=\mathbf{X}_m \bm{\beta}_m+\mathbf{X}_m^c \bm{\beta}_m^c$. It follows that
	\begin{equation}\label{eq:mu.m}
		\bm{\mu}^{\top} (\mathbf{P}_{M_0+1}-\mathbf{P}_m)\bm{\mu}=\bm{\mu}^{\top} (\mathbf{I}_n-\mathbf{P}_m)\bm{\mu}=\bm{\beta}_m^{c\top} (\mathbf{X}_m^{c\top}\mathbf{X}_m^c-\mathbf{X}_m^{c\top}\mathbf{P}_m\mathbf{X}_m^c) \bm{\beta}_m^c.
	\end{equation}
	Observe that
	\begin{equation}\label{eq:block}
		\mathbf{X}_{M_0+1}^{\top}\mathbf{X}_{M_0+1}=\begin{pmatrix}
			\mathbf{X}_m^{\top}\mathbf{X}_m & \mathbf{X}_m^{\top}\mathbf{X}_m^c \\
			\mathbf{X}_m^{c\top}\mathbf{X}_m & \mathbf{X}_m^{c\top}\mathbf{X}_m^c
		\end{pmatrix}=\mathbf{\Gamma}^{\top} \begin{pmatrix}
			\mathbf{X}_m^{\top}\mathbf{X}_m & \mathbf{0} \\
			\mathbf{0} & \mathbf{X}_m^{c\top}\mathbf{X}_m^c-\mathbf{X}_m^{c\top}\mathbf{P}_m\mathbf{X}_m^c
		\end{pmatrix}
		\mathbf{\Gamma},
	\end{equation}
	where
	\begin{equation*}
		\mathbf{\Gamma}=\begin{pmatrix}
			\mathbf{I}_{k_m} & (\mathbf{X}_m^{\top}\mathbf{X}_m)^{-1} \mathbf{X}_m^{\top}\mathbf{X}_m^c \\
			\mathbf{0} & \mathbf{I}_{k_{M_0+1}-k_m}
		\end{pmatrix}.
	\end{equation*}
	Therefore, $\mathbf{X}_m^{c\top}\mathbf{X}_m^c-\mathbf{X}_m^{c\top}\mathbf{P}_m\mathbf{X}_m^c$ is the Schur complement of the block $\mathbf{X}_m^{\top}\mathbf{X}_m$ in the matrix $\mathbf{X}_{M_0+1}^{\top}\mathbf{X}_{M_0+1}$. Let $\mathbf{z}=(\mathbf{z}_1^{\top}, \mathbf{z}_2^{\top})^{\top}$ be an $k_{M_0+1}\times 1$ vector, where $\mathbf{z}_1\in \mathbb{R}^{k_m}$ and $\mathbf{z}_2\in \mathbb{R}^{k_{M_0+1}-k_m}$. Denote $\mathcal{S}=\{\mathbf{z}\colon \mathbf{z}_1+(\mathbf{X}_m^{\top}\mathbf{X}_m)^{-1} \mathbf{X}_m^{\top}\mathbf{X}_m^c \mathbf{z}_2=\mathbf{0}, \mathbf{z}_2\neq \mathbf{0}\}$. Then using \eqref{eq:block}, we know $\mathbf{z}^{\top}(\mathbf{X}_{M_0+1}^{\top}\mathbf{X}_{M_0+1})\mathbf{z}=\mathbf{z}_2^{\top}(\mathbf{X}_m^{c\top}\mathbf{X}_m^c-\mathbf{X}_m^{c\top}\mathbf{P}_m\mathbf{X}_m^c)\mathbf{z}_2$ for any $\mathbf{z}\in \mathcal{S}$. It follows that
	\begin{align*}
		\lambda_{\min}(\mathbf{X}_{M_0+1}^{\top}\mathbf{X}_{M_0+1})&=\min_{\mathbf{z}\neq \mathbf{0}} \frac{\mathbf{z}^{\top}(\mathbf{X}_{M_0+1}^{\top}\mathbf{X}_{M_0+1})\mathbf{z}}{\mathbf{z}^{\top}\mathbf{z}} \\
		&\leq \min_{\mathbf{z}\in \mathcal{S}} \frac{\mathbf{z}^{\top}(\mathbf{X}_{M_0+1}^{\top}\mathbf{X}_{M_0+1})\mathbf{z}}{\mathbf{z}^{\top}\mathbf{z}} \\
		&=\min_{\mathbf{z}\in \mathcal{S}} \frac{\mathbf{z}_2^{\top}(\mathbf{X}_m^{c\top}\mathbf{X}_m^c-\mathbf{X}_m^{c\top}\mathbf{P}_m\mathbf{X}_m^c)\mathbf{z}_2}{\mathbf{z}_1^{\top}\mathbf{z}_1+\mathbf{z}_2^{\top}\mathbf{z}_2} \\
		&\leq \min_{\mathbf{z}_2\neq \mathbf{0}} \frac{\mathbf{z}_2^{\top}(\mathbf{X}_m^{c\top}\mathbf{X}_m^c-\mathbf{X}_m^{c\top}\mathbf{P}_m\mathbf{X}_m^c)\mathbf{z}_2}{\mathbf{z}_2^{\top}\mathbf{z}_2} \\
		&=\lambda_{\min}(\mathbf{X}_m^{c\top}\mathbf{X}_m^c-\mathbf{X}_m^{c\top}\mathbf{P}_m\mathbf{X}_m^c),
	\end{align*}
	which together with \eqref{eq:mu.m} and Condition \ref{cond:eigvalue}, yields that
	\begin{align*}
		\bm{\mu}^{\top} (\mathbf{P}_{M_0+1}-\mathbf{P}_m)\bm{\mu}&\geq \lambda_{\min}(\mathbf{X}_m^{c\top}\mathbf{X}_m^c-\mathbf{X}_m^{c\top}\mathbf{P}_m\mathbf{X}_m^c) \|\bm{\beta}_m^c\|^2 \\
		&\geq \lambda_{\min}(\mathbf{X}_{M_0+1}^{\top}\mathbf{X}_{M_0+1}) \|\bm{\beta}_m^c\|^2 \geq \kappa_0 n\|\bm{\beta}_m^c\|^2.
	\end{align*}
	This completes the proof of Lemma \ref{lem:mumu}.
\end{proof}

The following lemma establishes asymptotic distributions of some quantities if $k_{M_0+1}$ is fixed. We first introduce some notations. We partition $\mathbf{Q}$ given in Condition \ref{cond:xcong} as
\begin{equation*}
	\mathbf{Q}=\begin{pmatrix}
		\mathbf{Q}_{11} & \mathbf{Q}_{12} \\
		\mathbf{Q}_{21} & \mathbf{Q}_{22}
	\end{pmatrix},
\end{equation*}
where $\mathbf{Q}_{11}$ is an $k_{M_0}\times k_{M_0}$ submatrix, and $\mathbf{Q}_{12}^{\top}=\mathbf{Q}_{21}$ and $\mathbf{Q}_{22}$ are conformable with it for partitioning. Let
\begin{equation}\label{eq:unit}
	\mathbf{v}=\mathbf{Q}^{-1/2}
	\begin{pmatrix}
		\mathbf{0}_{k_{M_0}\times 1} \\
		(\mathbf{Q}_{22}-\mathbf{Q}_{21}\mathbf{Q}_{11}^{-1}\mathbf{Q}_{12}) \bm{\beta}_{M_0}^c
	\end{pmatrix} \big/ \left\{\bm{\beta}_{M_0}^{c\top}	(\mathbf{Q}_{22}-\mathbf{Q}_{21}\mathbf{Q}_{11}^{-1}\mathbf{Q}_{12}) \bm{\beta}_{M_0}^c\right\}^{1/2},
\end{equation}
where $\bm{\beta}_{M_0}^c=(\beta_j, j\in \mathcal{A}_{M_0+1}\cap \mathcal{A}_{M_0}^c)^{\top}$. Note that $\|\mathbf{v}\|=1$. When $\mathbf{Q}=\sigma_X^2 \mathbf{I}_{k_{M_0+1}}$ with some $\sigma_X^2>0$, it can be easily verified that $\mathbf{v}=(\mathbf{0}_{k_{M_0}\times 1}^{\top}, \bm{\beta}_{M_0}^{c\top}/\|\bm{\beta}_{M_0}^c\|)^{\top}$.

\begin{lemma}\label{lem:convg}
	If $k_{M_0+1}$ is fixed and Condition \ref{cond:xcong} is satisfied, then \begin{equation}\label{eq:cong1}
		\mathbf{e}^{\top}\mathbf{P}_{M_0+1}\mathbf{e}\to_d \sigma^2 \mathbf{Z}^{\top}\mathbf{Z}\sim \sigma^2\chi^2_{k_{M_0+1}},
	\end{equation}
	\begin{equation}\label{eq:cong2}
		\frac{\bm{\mu}^{\top}(\mathbf{P}_{M_0+1}-\mathbf{P}_{M_0})\mathbf{e}}{\sqrt{\bm{\mu}^{\top}(\mathbf{P}_{M_0+1}-\mathbf{P}_{M_0})\bm{\mu}}}\to_d \sigma \mathbf{v}^{\top}\mathbf{Z}\sim \sigma N(0,1),
	\end{equation}
	and
	\begin{equation}\label{eq:cong3}
		\frac{\{\bm{\mu}^{\top}(\mathbf{P}_{M_0+1}-\mathbf{P}_{M_0})\mathbf{e}\}^2}{\mathbf{e}^{\top}\mathbf{P}_{M_0+1}\mathbf{e}\{\bm{\mu}^{\top} (\mathbf{P}_{M_0+1}-\mathbf{P}_{M_0})\bm{\mu}\}}\to_d \frac{(\mathbf{v}^{\top}\mathbf{Z})^2}{\mathbf{Z}^{\top}\mathbf{Z}},
	\end{equation}
	where $\mathbf{Z}\sim N(\mathbf{0},\mathbf{I}_{k_{M_0+1}})$.
\end{lemma}
\begin{proof}
	Let $\mathbf{Q}_n=n^{-1}\mathbf{X}_{M_0+1}^{\top}\mathbf{X}_{M_0+1}$ and $\mathbf{Z}_n=n^{-1/2}\mathbf{Q}^{-1/2}\mathbf{X}_{M_0+1}^{\top} \mathbf{e}/\sigma$, where $\mathbf{Q}$ is defined in Condition \ref{cond:xcong}. Then, we have $\mathbf{e}^{\top}\mathbf{P}_{M_0+1}\mathbf{e}=\sigma^2 \mathbf{Z}_n^{\top}\mathbf{Q}^{1/2}\mathbf{Q}_n^{-1}\mathbf{Q}^{1/2}\mathbf{Z}_n$,
	\begin{align*}
		\bm{\mu}^{\top}(\mathbf{P}_{M_0+1}-\mathbf{P}_{M_0})\mathbf{e}&=\bm{\beta}_{M_0+1}^{\top} \left\{\mathbf{I}_{k_{M_0+1}}-\mathbf{Q}_n\bm{\Pi}_{M_0} \left(\bm{\Pi}_{M_0}^{\top}\mathbf{Q}_n\bm{\Pi}_{M_0}\right)^{-1}\bm{\Pi}_{M_0}^{\top}\right\}\mathbf{X}_{M_0+1}^{\top}\mathbf{e} \\
		&=n^{1/2}\sigma \bm{\beta}_{M_0+1}^{\top} \left\{\mathbf{I}_{k_{M_0+1}}-\mathbf{Q}_n\bm{\Pi}_{M_0} \left(\bm{\Pi}_{M_0}^{\top}\mathbf{Q}_n\bm{\Pi}_{M_0}\right)^{-1}\bm{\Pi}_{M_0}^{\top}\right\}\mathbf{Q}^{1/2}\mathbf{Z}_n,
	\end{align*}
	and
	\begin{equation*}
		\bm{\mu}^{\top}(\mathbf{P}_{M_0+1}-\mathbf{P}_{M_0})\bm{\mu}=n\bm{\beta}_{M_0+1}^{\top} \left\{\mathbf{Q}_n-\mathbf{Q}_n\bm{\Pi}_{M_0} \left(\bm{\Pi}_{M_0}^{\top}\mathbf{Q}_n\bm{\Pi}_{M_0}\right)^{-1} \bm{\Pi}_{M_0}^{\top}\mathbf{Q}_n\right\} \bm{\beta}_{M_0+1},
	\end{equation*}
	where $\bm{\beta}_{M_0+1}$ is defined in Lemma \ref{lem:mumu} and $\bm{\Pi}_{M_0}$ is a selection matrix such that $\mathbf{X}_{M_0+1} \bm{\Pi}_{M_0}=\mathbf{X}_{M_0}$. It follows from Condition \ref{cond:xcong} that $\mathbf{Q}_n\to \mathbf{Q}$ and $\mathbf{Z}_n\to_d \mathbf{Z}\sim N(\mathbf{0},\mathbf{I}_{k_{M_0+1}})$, which together with Slutsky's theorem and the continuous mapping theorem, yields that \eqref{eq:cong1},
	\begin{equation*}
		\frac{\bm{\mu}^{\top}(\mathbf{P}_{M_0+1}-\mathbf{P}_{M_0})\mathbf{e}}{\sqrt{\bm{\mu}^{\top}(\mathbf{P}_{M_0+1}-\mathbf{P}_{M_0})\bm{\mu}}}\to_d \sigma \frac{\mathbf{c}^{\top}\mathbf{Z}}{\sqrt{\mathbf{c}^{\top}\mathbf{c}}}\sim \sigma N(0,1),
	\end{equation*}
	and
	\begin{equation*}
		\frac{\{\bm{\mu}^{\top}(\mathbf{P}_{M_0+1}-\mathbf{P}_{M_0})\mathbf{e}\}^2}{\mathbf{e}^{\top}\mathbf{P}_{M_0+1}\mathbf{e}\{\bm{\mu}^{\top} (\mathbf{P}_{M_0+1}-\mathbf{P}_{M_0})\bm{\mu}\}}\to_d \frac{(\mathbf{c}^{\top}\mathbf{Z})^2/(\mathbf{c}^{\top}\mathbf{c})}{\mathbf{Z}^{\top}\mathbf{Z}},
	\end{equation*}
	where $\mathbf{c}= \mathbf{Q}^{1/2}\{\mathbf{I}_{k_{M_0+1}}-\bm{\Pi}_{M_0} (\bm{\Pi}_{M_0}^{\top}\mathbf{Q}\bm{\Pi}_{M_0})^{-1}\bm{\Pi}_{M_0}^{\top}\mathbf{Q}\}\bm{\beta}_{M_0+1}$. It is easy to verify that $\mathbf{c}/\sqrt{\mathbf{c}^{\top}\mathbf{c}}=\mathbf{v}$. This completes the proof of Lemma \ref{lem:convg}.
\end{proof}

The following lemma establishes the asymptotic properties of $\hat{w}_m$ for both under-fitted and over-fitted models. While Lemmas 1--2 of \citet{zhang2020PMA} have studied this problem, our results are more direct and provide simpler proofs.

\begin{lemma}\label{lem:what}
	(i) Under Conditions \ref{cond:eigvalue}--\ref{cond:sigma}, if $\phi_n (k_{M_0+1}-k_1)/(n\eta_n)\to 0$, then
	\begin{equation*}
		\sum_{m=1}^{M_0} \hat{w}_m=O_p\left\{\frac{\phi_n (k_{M_0+1}-k_1)}{n\eta_n}\right\}.
	\end{equation*}
	(ii) If Condition \ref{cond:sigma} is satisfied and $\phi_n\to \infty$, then
	\begin{equation*}
		\Pr\left(\sum_{m>M_0+1} \hat{w}_m=0\right)\to 1.
	\end{equation*}
\end{lemma}
\begin{proof}
	By using \eqref{eq:basic}, we can rewrite \eqref{eq:crit} as
	\begin{align*}
		\mathcal{G}_n(\mathbf{w})&=\left\|\mathbf{y}-\sum_{m=1}^{M_n} w_m \mathbf{P}_m \mathbf{y}\right\|^2+\phi_n\hat{\sigma}^2 \sum_{m=1}^{M_n} w_m k_m \\
		&=\left\|\mathbf{y}-\mathbf{P}_{M_0+1}\mathbf{y}+\sum_{m=1}^{M_n} w_m (\mathbf{P}_{M_0+1}-\mathbf{P}_m) \mathbf{y}\right\|^2+\phi_n\hat{\sigma}^2 \left\{k_{M_0+1}+ \sum_{m=1}^{M_n} w_m (k_m-k_{M_0+1})\right\} \\
		&=\left\|(\mathbf{I}_n-\mathbf{P}_{M_0+1})\mathbf{y}+\sum_{m=1}^{M_0}w_m (\mathbf{P}_{M_0+1}-\mathbf{P}_m)\mathbf{y}-\sum_{m>M_0+1} w_m (\mathbf{P}_m-\mathbf{P}_{M_0+1})\mathbf{y}\right\|^2 \\
		&\quad +\phi_n\hat{\sigma}^2 \left\{k_{M_0+1}-\sum_{m=1}^{M_0} w_m (k_{M_0+1}-k_m)+\sum_{m>M_0+1} w_m (k_m-k_{M_0+1})\right\} \\
		&=\mathcal{G}_n(\mathbf{w}_{M_0+1}^0)+I_{n,1}(\mathbf{w})+I_{n,2}(\mathbf{w}),
	\end{align*}
	where
	\begin{equation*}
		I_{n,1}(\mathbf{w})=\sum_{m=1}^{M_0} \left(w_m^2+2w_m\sum_{l=1}^{m-1} w_l\right) \mathbf{y}^{\top} (\mathbf{P}_{M_0+1}-\mathbf{P}_m)\mathbf{y}-\phi_n\hat{\sigma}^2 \sum_{m=1}^{M_0} w_m (k_{M_0+1}-k_m)
	\end{equation*}
	and
	\begin{align*}
		I_{n,2}(\mathbf{w})&=\sum_{m>M_0+1} \left(w_m^2+2w_m \sum_{l=m+1}^{M_n} w_l\right) \mathbf{e}^{\top}(\mathbf{P}_m-\mathbf{P}_{M_0+1})\mathbf{e} \\
		&\quad -\sum_{m>M_0+1} w_m \left\{2\mathbf{e}^{\top}(\mathbf{P}_m-\mathbf{P}_{M_0+1})\mathbf{e}-\phi_n\hat{\sigma}^2(k_m-k_{M_0+1})\right\}.
	\end{align*}
	
	\underline{Proof of part (i).} Denote $\tilde{\mathbf{w}}=(\mathbf{0}_{M_0\times 1}^{\top}, \sum_{m=1}^{M_0+1} \hat{w}_m, \hat{w}_{M_0+2},\ldots,\hat{w}_{M_n})^{\top}\in \mathcal{H}_n$. It follows from $\mathcal{G}_n(\hat{\mathbf{w}})\leq \mathcal{G}_n(\tilde{\mathbf{w}})$ that $I_{n,1}(\hat{\mathbf{w}})\leq 0$, i.e.,
	\begin{equation}\label{eq:ineq}
		\sum_{m=1}^{M_0} \left(\hat{w}_m^2+2\hat{w}_m\sum_{l=1}^{m-1} \hat{w}_l\right) \mathbf{y}^{\top} (\mathbf{P}_{M_0+1}-\mathbf{P}_m)\mathbf{y}\leq \phi_n\hat{\sigma}^2 \sum_{m=1}^{M_0} \hat{w}_m (k_{M_0+1}-k_m),
	\end{equation}
	which implies that $\mathbf{y}^{\top} (\mathbf{P}_{M_0+1}-\mathbf{P}_{M_0})\mathbf{y} (\sum_{m=1}^{M_0} \hat{w}_m)^2\leq \phi_n\hat{\sigma}^2 (k_{M_0+1}-k_1) \sum_{m=1}^{M_0} \hat{w}_m$. So when $\sum_{m=1}^{M_0} \hat{w}_m\neq 0$, we have
	\begin{align*}
		\sum_{m=1}^{M_0} \hat{w}_m&\leq \frac{\phi_n\hat{\sigma}^2 (k_{M_0+1}-k_1)}{\mathbf{y}^{\top} (\mathbf{P}_{M_0+1}-\mathbf{P}_{M_0})\mathbf{y}} \\
		 &=\frac{\phi_n\hat{\sigma}^2 (k_{M_0+1}-k_1)}{\bm{\mu}^{\top} (\mathbf{P}_{M_0+1}-\mathbf{P}_{M_0})\bm{\mu}+\mathbf{e}^{\top} (\mathbf{P}_{M_0+1}-\mathbf{P}_{M_0})\mathbf{e}+2\bm{\mu}^{\top}(\mathbf{P}_{M_0+1}-\mathbf{P}_{M_0})\mathbf{e}} \\
		 &\leq \frac{\phi_n\hat{\sigma}^2 (k_{M_0+1}-k_1)}{\kappa_0 n\eta_n+\mathbf{e}^{\top} (\mathbf{P}_{M_0+1}-\mathbf{P}_{M_0})\mathbf{e}+O_p\{(n\eta_n)^{1/2}\}} \\
		 &=O_p\left\{\frac{\phi_n (k_{M_0+1}-k_1)}{n\eta_n}\right\},
	\end{align*}
	where the second inequality follows from Lemma \ref{lem:mumu} and the last step is derived by Condition \ref{cond:sigma} and $\mathbf{e}^{\top} (\mathbf{P}_{M_0+1}-\mathbf{P}_{M_0})\mathbf{e}/(n\eta_n)=o_p(1)$. This completes the proof of part (i).
	
	\underline{Proof of part (ii).} Denote $\tilde{\mathbf{w}}^*=(\hat{w}_1,\ldots,\hat{w}_{M_0}, \sum_{m=M_0+1}^{M_n} \hat{w}_m, 0,\ldots,0)^{\top}\in \mathcal{H}_n$. It follows from $\mathcal{G}_n(\hat{\mathbf{w}})\leq \mathcal{G}_n(\tilde{\mathbf{w}}^*)$ that $I_{n,2}(\hat{\mathbf{w}})\leq 0$, i.e.,
	\begin{align}\label{eq:overfit}
		&\sum_{m>M_0+1} \left(\hat{w}_m^2+2\hat{w}_m \sum_{l=m+1}^{M_n} \hat{w}_l\right) \mathbf{e}^{\top}(\mathbf{P}_m-\mathbf{P}_{M_0+1})\mathbf{e} \notag \\
		&\quad \leq \sum_{m>M_0+1} \hat{w}_m \left\{2\mathbf{e}^{\top}(\mathbf{P}_m-\mathbf{P}_{M_0+1})\mathbf{e}-\phi_n\hat{\sigma}^2(k_m-k_{M_0+1})\right\}.
	\end{align}
	Define the event
	\begin{equation*}
		\mathfrak{A}_n=\left\{\mathbf{e}^{\top}(\mathbf{P}_m-\mathbf{P}_{M_0+1})\mathbf{e}< \phi_n\hat{\sigma}^2/2 (k_m-k_{M_0+1})~\text{for all}~ m>M_0+1\right\}.
	\end{equation*}
	We first prove that $\Pr(\mathfrak{A}_n)\to 1$ if $\phi_n\to \infty$. Let $\mathfrak{B}_n$ be the set of all realizations such that $c_1\leq \hat{\sigma}^2/\sigma^2\leq c_2$, where $c_1$ and $c_2$ are defined in Condition \ref{cond:sigma}. Observe that when $\mathfrak{B}_n$ holds,
	\begin{align}\label{eq:max.ineq}
		&\Pr\left\{\max_{m>M_0+1} \frac{\mathbf{e}^{\top}(\mathbf{P}_m-\mathbf{P}_{M_0+1})\mathbf{e}}{\sigma^2 (k_m-k_{M_0+1})}\geq \frac{\phi_n}{2}\frac{\hat{\sigma}^2}{\sigma^2}\right\} \notag\\
		&\quad \leq \Pr\left\{\max_{m>M_0+1} \frac{\mathbf{e}^{\top}(\mathbf{P}_m-\mathbf{P}_{M_0+1})\mathbf{e}}{\sigma^2 (k_m-k_{M_0+1})}\geq \frac{\phi_n}{2}c_1\right\} \notag\\
		&\quad \leq \sum_{m>M_0+1} \Pr\left\{\mathbf{e}^{\top}(\mathbf{P}_m-\mathbf{P}_{M_0+1})\mathbf{e}\geq \phi_n c_1\sigma^2/2 (k_m-k_{M_0+1})\right\} \notag\\
		&\quad = \sum_{m>M_0+1} \Pr\Bigl[\mathbf{e}^{\top}(\mathbf{P}_m-\mathbf{P}_{M_0+1})\mathbf{e}-E\{\mathbf{e}^{\top}(\mathbf{P}_m-\mathbf{P}_{M_0+1})\mathbf{e}\} \notag\\
		&\hspace{2.9cm}\geq (\phi_n c_1/2-1)\sigma^2 (k_m-k_{M_0+1})\Bigr] \notag\\
		&\quad \leq \sum_{m>M_0+1} \frac{E \left|\mathbf{e}^{\top}(\mathbf{P}_m-\mathbf{P}_{M_0+1})\mathbf{e}-E\{\mathbf{e}^{\top}(\mathbf{P}_m-\mathbf{P}_{M_0+1})\mathbf{e}\}\right|^{2+\zeta/2}}{(\phi_n c_1/2-1)^{2+\zeta/2}\sigma^{4+\zeta} (k_m-k_{M_0+1})^{2+\zeta/2}} \notag\\
		&\quad \leq \sum_{m>M_0+1} \frac{C_{\zeta}[\mathrm{tr}\{(\mathbf{P}_m-\mathbf{P}_{M_0+1})^2\}]^{1+\zeta/4}}{(\phi_n c_1/2-1)^{2+\zeta/2}\sigma^{4+\zeta} (k_m-k_{M_0+1})^{2+\zeta/2}} \notag\\
		&\quad =\frac{C_{\zeta}}{(\phi_n c_1/2-1)^{2+\zeta/2}\sigma^{4+\zeta}} \sum_{m>M_0+1}\frac{1}{(k_m-k_{M_0+1})^{1+\zeta/4}} \notag\\
		&\quad \leq \frac{C_{\zeta}}{(\phi_n c_1/2-1)^{2+\zeta/2}\sigma^{4+\zeta}} \sum_{j=1}^{\infty}\frac{1}{j^{1+\zeta/4}},
	\end{align}
	where the third inequality follows from Markov's inequality, the fourth inequality is derived by Whittle's inequality \citep[Theorem 2]{whittle1960}, $C_{\zeta}$ is the constant in Whittle's inequality, and the last inequality is a direct consequence of the fact that $k_m-k_{M_0+1}\geq m-(M_0+1)$ for any $m>M_0+1$. Combining \eqref{eq:max.ineq}, $\Pr(\mathfrak{B}_n)\to 1$ under Condition \ref{cond:sigma}, and the fact that $\sum_{j=1}^{\infty} j^{-1-\zeta/4}<\infty$, we have if $\phi_n\to \infty$,
	\begin{equation*}
		\Pr\left\{\max_{m>M_0+1} \frac{\mathbf{e}^{\top}(\mathbf{P}_m-\mathbf{P}_{M_0+1})\mathbf{e}}{\sigma^2 (k_m-k_{M_0+1})}\geq \frac{\phi_n}{2}\frac{\hat{\sigma}^2}{\sigma^2}\right\} \to 0.
	\end{equation*}
	Therefore, $\Pr(\mathfrak{A}_n)\to 1$ follows if $\phi_n\to \infty$. On the set $\mathfrak{A}_n$, by using \eqref{eq:overfit} and the non-negativity of $\hat{w}_m$, we can assert that $\sum_{m>M_0+1} \hat{w}_m=0$. As a result, if $\phi_n\to \infty$, $\Pr(\sum_{m>M_0+1} \hat{w}_m=0)\geq \Pr(\mathfrak{A}_n)\to 1$. This completes the proof of part (ii).
\end{proof}

\subsection{Proof of the Results in Example \ref{exam:toy}}\label{apped:examp1}

From the obtained results in Example \ref{exam:toy}, we have
\begin{align*}
	\frac{\inf_{\mathbf{w}\in \mathcal{H}_n} L_n(\mathbf{w})}{L_n(\mathbf{w}_2^0)}&=1-\frac{\{\mathbf{y}^{\top}(\mathbf{P}_2-\mathbf{P}_1)\mathbf{e}\}^2}{\mathbf{e}^{\top}\mathbf{P}_2\mathbf{e} \{\mathbf{y}^{\top}(\mathbf{P}_2-\mathbf{P}_1)\mathbf{y}\}} \mathbf{1}\{0<w_1^{\mathrm{opt}}<1\} \\
	&\quad +\frac{\mathbf{y}^{\top}(\mathbf{P}_2-\mathbf{P}_1)(\mathbf{y}-2\mathbf{e})}{\mathbf{e}^{\top}\mathbf{P}_2\mathbf{e}}\mathbf{1}\{w_1^{\mathrm{opt}}\geq 1\},
\end{align*}
where $\mathbf{1}\{\cdot\}$ is the indicator function. It follows that for a fixed constant $z\in [0, 1)$,
\begin{align}\label{eq:inf.prob}
	&\Pr\left\{\frac{\inf_{\mathbf{w}\in \mathcal{H}_n} L_n(\mathbf{w})}{L_n(\mathbf{w}_2^0)}\leq z\right\} \notag \\ 
	&\quad= \Pr\left[1-\frac{\{\mathbf{y}^{\top}(\mathbf{P}_2-\mathbf{P}_1)\mathbf{e}\}^2}{\mathbf{e}^{\top}\mathbf{P}_2\mathbf{e} \{\mathbf{y}^{\top}(\mathbf{P}_2-\mathbf{P}_1)\mathbf{y}\}}\leq z, 0<w_1^{\mathrm{opt}}<1\right] \notag \\
	&\qquad + \Pr\left\{1+\frac{\mathbf{y}^{\top}(\mathbf{P}_2-\mathbf{P}_1)(\mathbf{y}-2\mathbf{e})}{\mathbf{e}^{\top}\mathbf{P}_2\mathbf{e}} \leq z, w_1^{\mathrm{opt}}\geq 1\right\} \notag \\
	&\quad=: A_{n,1}+A_{n,2}.
\end{align}
Next, we analyze $A_{n,1}$ and $A_{n,2}$. First,
\begin{align*}
	\Pr(w_1^{\mathrm{opt}}\geq 1)&=\Pr\left\{ \mathbf{y}^{\top}(\mathbf{P}_2-\mathbf{P}_1)\mathbf{e}\geq \mathbf{y}^{\top}(\mathbf{P}_2-\mathbf{P}_1)\mathbf{y}\right\} \notag \\
	&=\Pr\left(n\beta_2^2+\beta_2\sum_{i=1}^n x_{i2}e_i\leq 0\right) \notag \\
	&=\Pr\left\{\beta_2 N(0,\sigma^2)\leq -\sqrt{n}\beta_2^2\right\} \to 0,
\end{align*}
which leads to $A_{n,2}\to 0$. Observe that $\mathbf{e}^{\top}\mathbf{P}_2\mathbf{e}=S_1^2+S_2^2$, where $S_1=n^{-1/2}\sum_{i=1}^n x_{i1}e_i$ and $S_2=n^{-1/2}\sum_{i=1}^n x_{i2}e_i$. It is easy to see that $S_1, S_2\sim N(0,\sigma^2)$ and they are independent. Therefore, we have
\begin{equation}\label{eq:beta}
	\frac{\{\mathbf{y}^{\top}(\mathbf{P}_2-\mathbf{P}_1)\mathbf{e}\}^2}{\mathbf{e}^{\top}\mathbf{P}_2\mathbf{e} \{\mathbf{y}^{\top}(\mathbf{P}_2-\mathbf{P}_1)\mathbf{y}\}}=\frac{S_2^4+n\beta_2^2 S_2^2+2\sqrt{n}\beta_2 S_2^3}{(S_1^2+S_2^2)(n\beta_2^2+S_2^2+2\sqrt{n}\beta_2 S_2)}\to_p \frac{S_2^2}{S_1^2+S_2^2}.
\end{equation}
It is easy to show that $0<w_1^{\mathrm{opt}}<1$ is equivalent to $\beta_2 S_2>0$, and $S_1^2/(S_1^2+S_2^2)\sim \mathrm{Beta}(1/2, 1/2)$, which together with \eqref{eq:beta}, yield that
\begin{align*}
	A_{n,1}&\to
	\begin{cases}
		\Pr\{S_2^2/(S_1^2+S_2^2)\geq 1-z, S_2>0\}, & \text{if}~\beta_2>0 \\
		\Pr\{S_2^2/(S_1^2+S_2^2)\geq 1-z, S_2<0\}, & \text{if}~\beta_2<0 \\
	\end{cases} \\
	&=\frac{1}{2}\Pr\left(\frac{S_1^2}{S_1^2+S_2^2}\leq z\right)=\frac{1}{2}\Pr\left\{ \mathrm{Beta}\left(\frac{1}{2},\frac{1}{2}\right)\leq z\right\}.
\end{align*}
Combining the above results of $A_{n,1}$ and $A_{n,2}$ with \eqref{eq:inf.prob}, we have
\begin{equation*}
	\Pr\left\{\frac{\inf_{\mathbf{w}\in \mathcal{H}_n} L_n(\mathbf{w})}{L_n(\mathbf{w}_2^0)}\leq z\right\}\to
	\begin{cases}
		0, & \text{if}~z< 0, \\
		2^{-1}\Pr\{ \mathrm{Beta}\left(1/2,1/2\right)\leq z\}, & \text{if}~0\leq z<1, \\
		1, & \text{if}~z\geq 1. \\
	\end{cases}
\end{equation*}
Finally, we obtain \eqref{eq:exmp1}.

\subsection{Proof of Lemma \ref{lem:loss}}\label{subsec:lem1}

For any $\mathbf{w}=(w_1,\ldots,w_{M_n})^{\top}\in \mathcal{H}_n$, denote $\bar{\mathbf{w}}=(w_1,\ldots,w_{M_0},\sum_{m=M_0+1}^{M_n} w_m, 0, \ldots,0)^{\top}\in \mathcal{H}_n$. Then, by the first part of Lemma \ref{lem:lossrisk}, we have
\begin{equation*}
	L_n(\mathbf{w})-L_n(\bar{\mathbf{w}})=\sum_{m>M_0+1} \left(w_m^2+2w_m \sum_{l=m+1}^{M_n} w_l\right) \mathbf{e}^{\top}(\mathbf{P}_m-\mathbf{P}_{M_0+1})\mathbf{e}\geq 0.
\end{equation*}
Therefore,
\begin{equation}\label{eq:inf.loss}
	\inf_{\mathbf{w}\in \mathcal{H}_n} L_n(\mathbf{w})=L_n(\mathbf{w}_{M_0+1}^0)+\inf_{(w_1,\ldots,w_{M_0})\in \mathcal{U}_n} F_n(w_1,\ldots,w_{M_0}),
\end{equation}
where $\mathcal{U}_n=\{(w_1,\ldots,w_{M_0})\colon w_m \geq 0; \sum_{m=1}^{M_0} w_m\leq 1\}$ and
\begin{equation*}
	F_n(w_1,\ldots,w_{M_0})=\sum_{m=1}^{M_0} \left(w_m^2+2w_m\sum_{l=1}^{m-1} w_l\right) \mathbf{y}^{\top} (\mathbf{P}_{M_0+1}-\mathbf{P}_m)\mathbf{y}-2\sum_{m=1}^{M_0} w_m \mathbf{y}^{\top} (\mathbf{P}_{M_0+1}-\mathbf{P}_m)\mathbf{e}.
\end{equation*}
Now we consider $k_{M_0+1}$ is fixed and $k_{M_0+1}\to \infty$, respectively.

\underline{(i) Consider $k_{M_0+1}$ is fixed.} First, we observe that
\begin{align}
	&\inf_{(w_1,\ldots,w_{M_0})\in \mathcal{U}_n} F_n(w_1,\ldots,w_{M_0}) \notag \\
	&\quad \leq \inf_{w_1=\cdots=w_{M_0-1}=0\leq w_{M_0}\leq 1} F_n(w_1,\ldots,w_{M_0}) \notag \\
	&\quad =\inf_{0\leq w_{M_0}\leq 1} \left\{w_{M_0}^2 \mathbf{y}^{\top} (\mathbf{P}_{M_0+1}-\mathbf{P}_{M_0})\mathbf{y}-2 w_{M_0} \mathbf{y}^{\top} (\mathbf{P}_{M_0+1}-\mathbf{P}_{M_0})\mathbf{e}\right\} \notag \\
	&\quad =-\frac{\{\mathbf{y}^{\top} (\mathbf{P}_{M_0+1}-\mathbf{P}_{M_0})\mathbf{e}\}^2}{\mathbf{y}^{\top} (\mathbf{P}_{M_0+1}-\mathbf{P}_{M_0})\mathbf{y}} \mathbf{1}\{0<w_{M_0}^{\text{opt}}<1\}+\mathbf{y}^{\top} (\mathbf{P}_{M_0+1}-\mathbf{P}_{M_0})(\mathbf{y}-2\mathbf{e})\mathbf{1}\{w_{M_0}^{\text{opt}}\geq 1\}, \label{eq:under.loss}
\end{align}
where
\begin{equation*}
	w_{M_0}^{\text{opt}}=\frac{\mathbf{y}^{\top} (\mathbf{P}_{M_0+1}-\mathbf{P}_{M_0})\mathbf{e}}{\mathbf{y}^{\top} (\mathbf{P}_{M_0+1}-\mathbf{P}_{M_0})\mathbf{y}}.
\end{equation*}
Then, it follows from Equations \eqref{eq:inf.loss} and \eqref{eq:under.loss} that
\begin{align*}
	\frac{\inf_{\mathbf{w}\in \mathcal{H}_n} L_n(\mathbf{w})}{L_n(\mathbf{w}_{M_0+1}^0)}&\leq 1-\frac{\{\mathbf{y}^{\top} (\mathbf{P}_{M_0+1}-\mathbf{P}_{M_0})\mathbf{e}\}^2}{\mathbf{e}^{\top}\mathbf{P}_{M_0+1}\mathbf{e}\{\mathbf{y}^{\top} (\mathbf{P}_{M_0+1}-\mathbf{P}_{M_0})\mathbf{y}\}} \mathbf{1}\{0<w_{M_0}^{\text{opt}}<1\} \notag \\
	&\quad +\frac{\mathbf{y}^{\top} (\mathbf{P}_{M_0+1}-\mathbf{P}_{M_0})(\mathbf{y}-2\mathbf{e})}{\mathbf{e}^{\top}\mathbf{P}_{M_0+1}\mathbf{e}}\mathbf{1}\{w_{M_0}^{\text{opt}}\geq 1\},
\end{align*}
where the equation holds if and only if $M_0=1$. Then for any constant $z>1$,
\begin{align}\label{eq:ratio}
 &\Pr\left\{\frac{L_n(\mathbf{w}_{M_0+1}^0)}{\inf_{\mathbf{w}\in \mathcal{H}_n} L_n(\mathbf{w})}\geq z\right\} \notag\\
 &\quad = \Pr\left\{\frac{\inf_{\mathbf{w}\in \mathcal{H}_n} L_n(\mathbf{w})}{L_n(\mathbf{w}_{M_0+1}^0)}\leq \frac{1}{z}\right\} \notag\\
&\quad \geq \Pr\left[\frac{\{\mathbf{y}^{\top} (\mathbf{P}_{M_0+1}-\mathbf{P}_{M_0})\mathbf{e}\}^2}{\mathbf{e}^{\top}\mathbf{P}_{M_0+1}\mathbf{e}\{\mathbf{y}^{\top} (\mathbf{P}_{M_0+1}-\mathbf{P}_{M_0})\mathbf{y}\}} \mathbf{1}\{0<w_{M_0}^{\text{opt}}<1\}\geq 1-\frac{1}{z}\right] \notag \\
&\qquad +\Pr\left[\frac{\mathbf{y}^{\top} (\mathbf{P}_{M_0+1}-\mathbf{P}_{M_0})(\mathbf{y}-2\mathbf{e})}{\mathbf{e}^{\top}\mathbf{P}_{M_0+1}\mathbf{e}}\mathbf{1}\{w_{M_0}^{\text{opt}}\geq 1\}\leq \frac{1}{z}-1\right].
\end{align}
Using Markov's inequality and Lemma \ref{lem:mumu}, we have
\begin{align*}
	\Pr(w_{M_0}^{\text{opt}}\geq 1)&=\Pr\left\{\mathbf{y}^{\top} (\mathbf{P}_{M_0+1}-\mathbf{P}_{M_0})\mathbf{e}\geq \mathbf{y}^{\top} (\mathbf{P}_{M_0+1}-\mathbf{P}_{M_0})\mathbf{y}\right\} \\
	&=\Pr\left\{-\bm{\mu}^{\top} (\mathbf{P}_{M_0+1}-\mathbf{P}_{M_0})\mathbf{e}\geq \bm{\mu}^{\top} (\mathbf{P}_{M_0+1}-\mathbf{P}_{M_0})\bm{\mu}\right\} \\
	&\leq \frac{E\{\bm{\mu}^{\top} (\mathbf{P}_{M_0+1}-\mathbf{P}_{M_0})\mathbf{e}\}^2}{\{\bm{\mu}^{\top} (\mathbf{P}_{M_0+1}-\mathbf{P}_{M_0})\bm{\mu}\}^2} \\
	&= \frac{\sigma^2}{\bm{\mu}^{\top} (\mathbf{P}_{M_0+1}-\mathbf{P}_{M_0})\bm{\mu}} \\
	&\leq \frac{\sigma^2}{\kappa_0 n\eta_n} \to 0,
\end{align*}
where the last step follows from $\eta_n^{-1}=O(1)$ for the fixed true model dimension. It follows that $\Pr(w_{M_0}^{\text{opt}}\geq 1)\to 0$, which implies
\begin{equation}\label{eq:term1}
\Pr\left[\frac{\mathbf{y}^{\top} (\mathbf{P}_{M_0+1}-\mathbf{P}_{M_0})(\mathbf{y}-2\mathbf{e})}{\mathbf{e}^{\top}\mathbf{P}_{M_0+1}\mathbf{e}}\mathbf{1}\{w_{M_0}^{\text{opt}}\geq 1\}\leq \frac{1}{z}-1\right]\to 0.
\end{equation}
For the second term of \eqref{eq:ratio},
\begin{align}\label{eq:disfun}
	&\Pr\left[\frac{\{\mathbf{y}^{\top} (\mathbf{P}_{M_0+1}-\mathbf{P}_{M_0})\mathbf{e}\}^2}{\mathbf{e}^{\top}\mathbf{P}_{M_0+1}\mathbf{e}\{\mathbf{y}^{\top} (\mathbf{P}_{M_0+1}-\mathbf{P}_{M_0})\mathbf{y}\}} \mathbf{1}\{0<w_{M_0}^{\text{opt}}<1\}\geq 1-\frac{1}{z}\right] \notag \\
	&\quad =\Pr\left[\frac{\{\mathbf{y}^{\top} (\mathbf{P}_{M_0+1}-\mathbf{P}_{M_0})\mathbf{e}\}^2}{\mathbf{e}^{\top}\mathbf{P}_{M_0+1}\mathbf{e}\{\mathbf{y}^{\top} (\mathbf{P}_{M_0+1}-\mathbf{P}_{M_0})\mathbf{y}\}}\geq 1-\frac{1}{z}, 0<w_{M_0}^{\text{opt}}<1\right] \notag \\
	&\quad =\Pr\left[\frac{\{\mathbf{y}^{\top} (\mathbf{P}_{M_0+1}-\mathbf{P}_{M_0})\mathbf{e}\}^2}{\mathbf{e}^{\top}\mathbf{P}_{M_0+1}\mathbf{e}\{\mathbf{y}^{\top} (\mathbf{P}_{M_0+1}-\mathbf{P}_{M_0})\mathbf{y}\}}\geq 1-\frac{1}{z}, w_{M_0}^{\text{opt}}>0\right] \notag \\
	&\qquad -\Pr\left[\frac{\{\mathbf{y}^{\top} (\mathbf{P}_{M_0+1}-\mathbf{P}_{M_0})\mathbf{e}\}^2}{\mathbf{e}^{\top}\mathbf{P}_{M_0+1}\mathbf{e}\{\mathbf{y}^{\top} (\mathbf{P}_{M_0+1}-\mathbf{P}_{M_0})\mathbf{y}\}}\geq 1-\frac{1}{z}, w_{M_0}^{\text{opt}}\geq 1\right] \notag \\
	&\quad=: A_{n,3}+A_{n,4}.
\end{align}
Since $\Pr(w_{M_0}^{\text{opt}}\geq 1)\to 0$, we have $A_{n,4}\to 0$. Observe that
\begin{equation}\label{eq:asy}
	\frac{\{\mathbf{y}^{\top} (\mathbf{P}_{M_0+1}-\mathbf{P}_{M_0})\mathbf{e}\}^2}{\mathbf{e}^{\top}\mathbf{P}_{M_0+1}\mathbf{e}\{\mathbf{y}^{\top} (\mathbf{P}_{M_0+1}-\mathbf{P}_{M_0})\mathbf{y}\}}=\frac{\{\bm{\mu}^{\top}(\mathbf{P}_{M_0+1}-\mathbf{P}_{M_0})\mathbf{e}\}^2}{\mathbf{e}^{\top}\mathbf{P}_{M_0+1}\mathbf{e}\{\bm{\mu}^{\top} (\mathbf{P}_{M_0+1}-\mathbf{P}_{M_0})\bm{\mu}\}} B_{n,1}+B_{n,2},
\end{equation}
where
\begin{equation*}
	B_{n,1}=\frac{1}{1+\frac{\mathbf{e}^{\top} (\mathbf{P}_{M_0+1}-\mathbf{P}_{M_0})\mathbf{e}}{\bm{\mu}^{\top} (\mathbf{P}_{M_0+1}-\mathbf{P}_{M_0})\bm{\mu}}+\frac{2\bm{\mu}^{\top}(\mathbf{P}_{M_0+1}-\mathbf{P}_{M_0})\mathbf{e}}{\bm{\mu}^{\top} (\mathbf{P}_{M_0+1}-\mathbf{P}_{M_0})\bm{\mu}}}
\end{equation*}
and
\begin{equation*}
	B_{n,2}=\frac{\{\mathbf{e}^{\top} (\mathbf{P}_{M_0+1}-\mathbf{P}_{M_0})\mathbf{e}\}^2+2\{\bm{\mu}^{\top}(\mathbf{P}_{M_0+1}-\mathbf{P}_{M_0})\mathbf{e}\}\{\mathbf{e}^{\top} (\mathbf{P}_{M_0+1}-\mathbf{P}_{M_0})\mathbf{e}\}}{\mathbf{e}^{\top}\mathbf{P}_{M_0+1}\mathbf{e}\{\bm{\mu}^{\top} (\mathbf{P}_{M_0+1}-\mathbf{P}_{M_0})\bm{\mu}\}} B_{n,1}.
\end{equation*}
By using Lemma \ref{lem:mumu}, $\bm{\mu}^{\top} (\mathbf{P}_{M_0+1}-\mathbf{P}_{M_0})\bm{\mu}\geq \kappa_0 n\eta_n\to \infty$, which together with \eqref{eq:cong1} and \eqref{eq:cong2}, leads to $B_{n,1}\to_p 1$ and $B_{n,2}\to_p 0$. Combining it with \eqref{eq:cong2}, \eqref{eq:cong3}, \eqref{eq:asy}, and Slutsky's theorem, we have
\begin{equation*}
	\begin{pmatrix}
	\frac{\{\mathbf{y}^{\top} (\mathbf{P}_{M_0+1}-\mathbf{P}_{M_0})\mathbf{e}\}^2}{\mathbf{e}^{\top}\mathbf{P}_{M_0+1}\mathbf{e}\{\mathbf{y}^{\top} (\mathbf{P}_{M_0+1}-\mathbf{P}_{M_0})\mathbf{y}\}} \\
	\frac{\mathbf{y}^{\top} (\mathbf{P}_{M_0+1}-\mathbf{P}_{M_0})\mathbf{e}}{\sqrt{\bm{\mu}^{\top} (\mathbf{P}_{M_0+1}-\mathbf{P}_{M_0})\bm{\mu}}}
	\end{pmatrix}
	\to_d
	\begin{pmatrix} \frac{(\mathbf{v}^{\top}\mathbf{Z})^2}{\mathbf{Z}^{\top}\mathbf{Z}} \\
	\sigma \mathbf{v}^{\top}\mathbf{Z}
	\end{pmatrix},
\end{equation*}
where $\mathbf{c}$ and $\mathbf{Z}$ are defined in Lemma \ref{lem:convg}. Therefore,
\begin{align*}
	A_{n,3}&=\Pr\left[\frac{\{\mathbf{y}^{\top} (\mathbf{P}_{M_0+1}-\mathbf{P}_{M_0})\mathbf{e}\}^2}{\mathbf{e}^{\top}\mathbf{P}_{M_0+1}\mathbf{e}\{\mathbf{y}^{\top} (\mathbf{P}_{M_0+1}-\mathbf{P}_{M_0})\mathbf{y}\}}\geq 1-\frac{1}{z}, \mathbf{y}^{\top} (\mathbf{P}_{M_0+1}-\mathbf{P}_{M_0})\mathbf{e}>0\right] \\
	&=\Pr\left[\frac{\{\mathbf{y}^{\top} (\mathbf{P}_{M_0+1}-\mathbf{P}_{M_0})\mathbf{e}\}^2}{\mathbf{e}^{\top}\mathbf{P}_{M_0+1}\mathbf{e}\{\mathbf{y}^{\top} (\mathbf{P}_{M_0+1}-\mathbf{P}_{M_0})\mathbf{y}\}}\geq 1-\frac{1}{z}, \frac{\mathbf{y}^{\top} (\mathbf{P}_{M_0+1}-\mathbf{P}_{M_0})\mathbf{e}}{\sqrt{\bm{\mu}^{\top} (\mathbf{P}_{M_0+1}-\mathbf{P}_{M_0})\bm{\mu}}}>0\right] \\
	&\to \Pr\left\{ \frac{(\mathbf{v}^{\top}\mathbf{Z})^2}{\mathbf{Z}^{\top}\mathbf{Z}}\geq 1-\frac{1}{z}, \mathbf{v}^{\top}\mathbf{Z}>0\right\} \\
	&=\frac{1}{2} \Pr\left\{ \frac{(\mathbf{v}^{\top}\mathbf{Z})^2}{\mathbf{Z}^{\top}\mathbf{Z}}\geq 1-\frac{1}{z}\right\}.
\end{align*}
Combining the results of $A_{n,3}$ and $A_{n,4}$ with \eqref{eq:disfun}, we have
\begin{align}\label{eq:term2}
&\Pr\left[\frac{\{\mathbf{y}^{\top} (\mathbf{P}_{M_0+1}-\mathbf{P}_{M_0})\mathbf{e}\}^2}{\mathbf{e}^{\top}\mathbf{P}_{M_0+1}\mathbf{e}\{\mathbf{y}^{\top} (\mathbf{P}_{M_0+1}-\mathbf{P}_{M_0})\mathbf{y}\}} \mathbf{1}\{0<w_{M_0}^{\text{opt}}<1\}\geq 1-\frac{1}{z}\right] \notag \\
&\quad \to \frac{1}{2} \Pr\left\{ \frac{(\mathbf{v}^{\top}\mathbf{Z})^2}{\mathbf{Z}^{\top}\mathbf{Z}}\geq 1-\frac{1}{z}\right\}.
\end{align}
As a result, using \eqref{eq:ratio}, \eqref{eq:term1}, and \eqref{eq:term2}, we have
\begin{equation*}
	\liminf_{n\to\infty} \Pr\left\{\frac{L_n(\mathbf{w}_{M_0+1}^0)}{\inf_{\mathbf{w}\in \mathcal{H}_n} L_n(\mathbf{w})}\geq z\right\}\geq \frac{1}{2} \Pr\left\{ \frac{(\mathbf{v}^{\top}\mathbf{Z})^2}{\mathbf{Z}^{\top}\mathbf{Z}}\geq 1-\frac{1}{z}\right\}.
\end{equation*}
Finally, \eqref{eq:infp} holds and $L_n(\mathbf{w}_{M_0+1}^0)/\inf_{\mathbf{w}\in \mathcal{H}_n} L_n(\mathbf{w})\nrightarrow_p 1$. This completes the proof of Lemma \ref{lem:loss}(i).

\underline{(ii) Consider $k_{M_0+1}\to \infty$.} First, we observe that
\begin{align}\label{eq:divgf}
	&\inf_{(w_1,\ldots,w_{M_0})\in \mathcal{U}_n} F_n(w_1,\ldots,w_{M_0}) \notag \\
	&\quad \geq \inf_{(w_1,\ldots,w_{M_0})\in \mathcal{U}_n} \left\{\left(\sum_{m=1}^{M_0} w_m\right)^2 \mathbf{y}^{\top} (\mathbf{P}_{M_0+1}-\mathbf{P}_{M_0})\mathbf{y}\right. \notag \\
	&\hspace{3.3cm}\left.-\left(\sum_{m=1}^{M_0} w_m\right) 2\max_{1\leq m\leq M_0} \left|\mathbf{y}^{\top} (\mathbf{P}_{M_0+1}-\mathbf{P}_m)\mathbf{e}\right|\right\} \notag \\
	&\quad \geq -\frac{\left\{\max_{1\leq m\leq M_0} \left|\mathbf{y}^{\top} (\mathbf{P}_{M_0+1}-\mathbf{P}_m)\mathbf{e}\right|\right\}^2}{\mathbf{y}^{\top} (\mathbf{P}_{M_0+1}-\mathbf{P}_{M_0})\mathbf{y}} \notag \\
	&\quad \geq -\frac{\left\{\max_{1\leq m\leq M_0} \left|\bm{\mu}^{\top} (\mathbf{P}_{M_0+1}-\mathbf{P}_m)\mathbf{e}\right|+\mathbf{e}^{\top} (\mathbf{P}_{M_0+1}-\mathbf{P}_1)\mathbf{e}\right\}^2}{\bm{\mu}^{\top} (\mathbf{P}_{M_0+1}-\mathbf{P}_{M_0})\bm{\mu}+\mathbf{e}^{\top} (\mathbf{P}_{M_0+1}-\mathbf{P}_{M_0})\mathbf{e}+2\bm{\mu}^{\top}(\mathbf{P}_{M_0+1}-\mathbf{P}_{M_0})\mathbf{e}}.
\end{align}
For any $C_{\epsilon}>0$, it follows from Markov's inequality and Condition \ref{cond:mu} that
\begin{align*}
	&\Pr\left\{\max_{1\leq m\leq M_0} \left|\bm{\mu}^{\top} (\mathbf{P}_{M_0+1}-\mathbf{P}_m)\mathbf{e}\right|\geq (n M_0)^{1/2}C_{\epsilon}\right\} \\
	&\quad \leq \sum_{m=1}^{M_0} \Pr\left\{\left|\bm{\mu}^{\top} (\mathbf{P}_{M_0+1}-\mathbf{P}_m)\mathbf{e}\right|\geq (n M_0)^{1/2} C_{\epsilon}\right\} \\
	&\quad \leq \sum_{m=1}^{M_0} \{(n M_0)^{1/2}C_{\epsilon}\}^{-2} E\left\{\bm{\mu}^{\top} (\mathbf{P}_{M_0+1}-\mathbf{P}_m)\mathbf{e}\right\}^2 \\
	&\quad \leq \sigma^2\sum_{m=1}^{M_0} \{(n M_0)^{1/2}C_{\epsilon}\}^{-2} \bm{\mu}^{\top} (\mathbf{P}_{M_0+1}-\mathbf{P}_m)\bm{\mu} \\
	&\quad \leq \sigma^2 M_0 \{(n M_0)^{1/2}C_{\epsilon}\}^{-2} \|\bm{\mu}\|^2=O(C_{\epsilon}^{-2}),
\end{align*}
which yields that
\begin{equation}\label{eq:max.mue}
	\max_{1\leq m\leq M_0} \left|\bm{\mu}^{\top} (\mathbf{P}_{M_0+1}-\mathbf{P}_m)\mathbf{e}\right|=O_p\left\{(n M_0)^{1/2}\right\}.
\end{equation}
By using Whittle's inequality \citep[Theorem 2]{whittle1960} and Chebyshev's inequality, it is easy to show that
\begin{equation}\label{eq:quadform}
	\mathbf{e}^{\top}\mathbf{P}_{M_0+1}\mathbf{e}/(\sigma^2 k_{M_0+1})\to_p 1\quad \text{as}~k_{M_0+1}\to \infty.
\end{equation}
Combining \eqref{eq:inf.loss}, \eqref{eq:divgf}--\eqref{eq:quadform}, and Lemma \ref{lem:mumu}, we obtain
\begin{equation}\label{eq:Lratio}
	\frac{\inf_{\mathbf{w}\in \mathcal{H}_n} L_n(\mathbf{w})}{L_n(\mathbf{w}_{M_0+1}^0)}\geq 1-\frac{\left[O_p\{(n M_0)^{1/2}\}+\mathbf{e}^{\top} (\mathbf{P}_{M_0+1}-\mathbf{P}_1)\mathbf{e}\right]^2}{\mathbf{e}^{\top}\mathbf{P}_{M_0+1}\mathbf{e}[\kappa_0 n\eta_n+\mathbf{e}^{\top} (\mathbf{P}_{M_0+1}-\mathbf{P}_{M_0})\mathbf{e}+O_p\{(n\eta_n)^{1/2}\}]}\to_p 1,
\end{equation}
where ``$\to_p$" follows from the conditions $(k_{M_0+1}-k_1)/(n\eta_n)\to 0$ and $M_0/(k_{M_0+1}\eta_n)\to 0$. This completes the proof of Lemma \ref{lem:loss}(ii).

\subsection{Proof of the Results in Remark \ref{rem:special.w}}\label{subsec:special.w}

Using similar arguments as in the proof of \eqref{eq:inf.loss}, we obtain
\begin{align}\label{eq:inf.dist}
	&\inf_{\mathbf{w}\in \mathcal{H}_n(N)} L_n(\mathbf{w})-L_n(\mathbf{w}_{M_0+1}^0)=\inf_{(w_1,\ldots,w_{M_0})\in \mathcal{U}_n(N)} F_n(w_1,\ldots,w_{M_0}) \notag \\
	&\quad \geq \inf_{(w_1,\ldots,w_{M_0})\in \mathcal{U}_n(N)} \left\{\left(\sum_{m=1}^{M_0} w_m\right)^2 \mathbf{y}^{\top} (\mathbf{P}_{M_0+1}-\mathbf{P}_{M_0})\mathbf{y}\right. \notag \\
	&\hspace{3.8cm} \left.-\left(\sum_{m=1}^{M_0} w_m\right) 2\max_{1\leq m\leq M_0} \left|\mathbf{y}^{\top} (\mathbf{P}_{M_0+1}-\mathbf{P}_m)\mathbf{e}\right|\right\},
\end{align}
where $\mathcal{U}_n(N)=\{(w_1,\ldots,w_{M_0})\colon w_m \in \{0, 1/N, 2/N, \ldots, 1\}; \sum_{m=1}^{M_0} w_m\leq 1\}$. In this case, $\sum_{m=1}^{M_0} w_m\in \{0, 1/N, 2/N, \ldots, 1\}$. Define
\begin{equation*}
	w_{\mathrm{sum}}^*=\frac{\max_{1\leq m\leq M_0} \left|\mathbf{y}^{\top} (\mathbf{P}_{M_0+1}-\mathbf{P}_m)\mathbf{e}\right|}{\mathbf{y}^{\top} (\mathbf{P}_{M_0+1}-\mathbf{P}_{M_0})\mathbf{y}}.
\end{equation*}
It is easy to see that when $w_{\mathrm{sum}}^*<1/(2N)$, the third infimum in \eqref{eq:inf.dist} is achieved if $\sum_{m=1}^{M_0} w_m=0$, and thus $\inf_{\mathbf{w}\in \mathcal{H}_n(N)} L_n(\mathbf{w})=L_n(\mathbf{w}_{M_0+1}^0)$. It follows that
\begin{equation}\label{eq:disw}
	\Pr\left\{\inf_{\mathbf{w}\in \mathcal{H}_n(N)} L_n(\mathbf{w})=L_n(\mathbf{w}_{M_0+1}^0)\right\}\geq \Pr\left(w_{\mathrm{sum}}^*<\frac{1}{2N}\right).
\end{equation}
By similar arguments to the proof of \eqref{eq:Lratio}, we can show that if $(k_{M_0+1}-k_1)/(n\eta_n)\to 0$ and $M_0/(n\eta_n^2)\to 0$, then $w_{\mathrm{sum}}^*\to_p 0$ and $\Pr\{w_{\mathrm{sum}}^*\geq 1/(2N)\}\to 0$, which along with \eqref{eq:disw} yields that $\Pr\{\inf_{\mathbf{w}\in \mathcal{H}_n(N)} L_n(\mathbf{w})=L_n(\mathbf{w}_{M_0+1}^0)\}\to 1$.

For the special weight set $\mathcal{H}_n^{\delta}$, similar to the proof of \eqref{eq:inf.dist}, we have
\begin{align*}
	&\inf_{\mathbf{w}\in \mathcal{H}_n^{\delta}} L_n(\mathbf{w})-L_n(\mathbf{w}_{M_0+1}^0) \notag \\
	&\quad \geq \inf_{\sum_{m=1}^{M_0} w_m=0~\text{or}~\sum_{m=1}^{M_0} w_m\geq \delta n^{-\tau_0}} \left\{\left(\sum_{m=1}^{M_0} w_m\right)^2 \mathbf{y}^{\top} (\mathbf{P}_{M_0+1}-\mathbf{P}_{M_0})\mathbf{y}\right. \\
	&\hspace{5.9cm} \left.-\left(\sum_{m=1}^{M_0} w_m\right) 2\max_{1\leq m\leq M_0} \left|\mathbf{y}^{\top} (\mathbf{P}_{M_0+1}-\mathbf{P}_m)\mathbf{e}\right|\right\},
\end{align*}
where when $w_{\mathrm{sum}}^*<\delta n^{-\tau_0}/2$, the infimum is achieved if $\sum_{m=1}^{M_0} w_m=0$. Therefore,
\begin{equation*}
	\Pr\left\{\inf_{\mathbf{w}\in \mathcal{H}_n^{\delta}} L_n(\mathbf{w})=L_n(\mathbf{w}_{M_0+1}^0)\right\}\geq \Pr\left(w_{\mathrm{sum}}^*<\frac{1}{2}\delta n^{-\tau_0}\right).
\end{equation*}
It can be easily shown that if $(k_{M_0+1}-k_1)/n^{1-2\tau_0}\to 0$ and $M_0/n^{1-4\tau_0}\to 0$, then $n^{\tau_0}w_{\mathrm{sum}}^*\to_p 0$ and $\Pr(w_{\mathrm{sum}}^*\geq \delta n^{-\tau_0}/2)\to 0$, which leads to $\Pr\{\inf_{\mathbf{w}\in \mathcal{H}_n^{\delta}} L_n(\mathbf{w})=L_n(\mathbf{w}_{M_0+1}^0)\}\to 1$.

\subsection{Proof of Theorem \ref{thm:opt.loss}}\label{subsec:thm1}

From the first part of Lemma \ref{lem:lossrisk}, we have $L_n(\hat{\mathbf{w}})=L_n(\mathbf{w}_{M_0+1}^0)+II_{n,1}-2II_{n,2}+II_{n,3}$,
where
\begin{align*}
	&II_{n,1}=\sum_{m=1}^{M_0} \left(\hat{w}_m^2+2\hat{w}_m\sum_{l=1}^{m-1} \hat{w}_l\right) \mathbf{y}^{\top} (\mathbf{P}_{M_0+1}-\mathbf{P}_m)\mathbf{y}, \\
	&II_{n,2}=\sum_{m=1}^{M_0} \hat{w}_m \mathbf{y}^{\top} (\mathbf{P}_{M_0+1}-\mathbf{P}_m)\mathbf{e},
	\intertext{and}
	&II_{n,3}=\sum_{m>M_0+1} \left(\hat{w}_m^2+2\hat{w}_m \sum_{l=m+1}^{M_n} \hat{w}_l\right) \mathbf{e}^{\top}(\mathbf{P}_m-\mathbf{P}_{M_0+1})\mathbf{e}.
\end{align*}
It follows from Inequality \eqref{eq:ineq} and Lemma \ref{lem:what} that
\begin{equation*}
	II_{n,1}\leq \phi_n\hat{\sigma}^2 \sum_{m=1}^{M_0} \hat{w}_m (k_{M_0+1}-k_m)\leq \phi_n\hat{\sigma}^2 (k_{M_0+1}-k_1)\sum_{m=1}^{M_0} \hat{w}_m=O_p\left\{\frac{\phi_n^2 (k_{M_0+1}-k_1)^2}{n\eta_n}\right\}.
\end{equation*}
By using \eqref{eq:max.mue} and Lemma \ref{lem:what}, we have
\begin{align*}
	|II_{n,2}|&\leq \max_{1\leq m\leq M_0} \left|\mathbf{y}^{\top} (\mathbf{P}_{M_0+1}-\mathbf{P}_m)\mathbf{e}\right| \sum_{m=1}^{M_0} \hat{w}_m \\
	&\leq \left\{\max_{1\leq m\leq M_0} \left|\bm{\mu}^{\top} (\mathbf{P}_{M_0+1}-\mathbf{P}_m)\mathbf{e}\right|+\mathbf{e}^{\top} (\mathbf{P}_{M_0+1}-\mathbf{P}_1)\mathbf{e}\right\} \sum_{m=1}^{M_0} \hat{w}_m \\
	&=\left[O_p\{(n M_0)^{1/2}\}+O_p(k_{M_0+1}-k_1)\right] O_p\left\{\frac{\phi_n (k_{M_0+1}-k_1)}{n\eta_n}\right\} \\
	&=O_p\left\{\frac{\phi_n M_0^{1/2} (k_{M_0+1}-k_1)}{n^{1/2}\eta_n}\right\}+O_p\left\{\frac{\phi_n (k_{M_0+1}-k_1)^2}{n\eta_n}\right\}.
\end{align*}
From \eqref{eq:cong1} and \eqref{eq:quadform}, we know that $1/L_n(\mathbf{w}_{M_0+1}^0)=1/\mathbf{e}^{\top}\mathbf{P}_{M_0+1}\mathbf{e}=O_p(k_{M_0+1}^{-1})$. Therefore, if $\phi_n^2(k_{M_0+1}-k_1)^2/(k_{M_0+1} n\eta_n)\to 0$ and $\phi_n^2 M_0 (k_{M_0+1}-k_1)^2/(k_{M_0+1}^2 n\eta_n^2)\to 0$, we have $II_{n,1}/L_n(\mathbf{w}_{M_0+1}^0)=o_p(1)$ and $II_{n,2}/L_n(\mathbf{w}_{M_0+1}^0)=o_p(1)$. As a result,
\begin{equation}\label{eq:loss.h.t}
	\frac{L_n(\hat{\mathbf{w}})}{L_n(\mathbf{w}_{M_0+1}^0)}=1+o_p(1)+\frac{II_{n,3}}{\mathbf{e}^{\top}\mathbf{P}_{M_0+1}\mathbf{e}}.
\end{equation}
Next, we consider $\phi_n\to \infty$ and $\phi_n=2$, respectively.

\underline{(i) Consider $\phi_n\to \infty$.} Define the event $\mathfrak{F}_n=\{\sum_{m>M_0+1} \hat{w}_m=0\}$. On the set $\mathfrak{F}_n$, $II_{n,3}=0$, which along with \eqref{eq:loss.h.t} and the fact that $\Pr(\mathfrak{F}_n)\to 1$ from Lemma \ref{lem:what}(ii), yields that
\begin{equation}\label{eq:hat.true}
	\frac{L_n(\hat{\mathbf{w}})}{L_n(\mathbf{w}_{M_0+1}^0)}=1+o_p(1).
\end{equation}
If $k_{M_0+1}$ is fixed and $\phi_n^2/n \to 0$, it follows from Lemma \ref{lem:loss}(i) and \eqref{eq:hat.true} that the asymptotic loss optimality does not hold. Furthermore, using the same argument for deriving \eqref{eq:infp}, \eqref{eq:hat.true}, and Slutsky's theorem, it can be easily proved that
\begin{align*}
	\liminf_{n\to\infty} \Pr\left\{\frac{L_n(\hat{\mathbf{w}})}{\inf_{\mathbf{w}\in \mathcal{H}_n} L_n(\mathbf{w})}\geq z\right\}&=\liminf_{n\to\infty} \Pr\left\{\frac{L_n(\mathbf{w}_{M_0+1}^0)}{L_n(\hat{\mathbf{w}})}\times \frac{\inf_{\mathbf{w}\in \mathcal{H}_n} L_n(\mathbf{w})}{L_n(\mathbf{w}_{M_0+1}^0)}\leq \frac{1}{z}\right\} \\
	&\geq \frac{1}{2} \Pr\left\{ \frac{(\mathbf{v}^{\top}\mathbf{Z})^2}{\|\mathbf{Z}\|^2}\geq 1-\frac{1}{z}\right\}.
\end{align*}
Finally, \eqref{eq:infpw} holds. Observe that $\phi_n^2(k_{M_0+1}-k_1)/(n\eta_n)\to 0$ and $M_0/(k_{M_0+1}\eta_n)\to 0$ imply that $\phi_n^2 M_0 (k_{M_0+1}-k_1)^2/(k_{M_0+1}^2 n\eta_n^2)\to 0$. Consequently, if $k_{M_0+1}\to \infty$, $\phi_n^2(k_{M_0+1}-k_1)/(n\eta_n)\to 0$, and $M_0/(k_{M_0+1}\eta_n)\to 0$, we can deduce from Lemma \ref{lem:loss}(ii) and \eqref{eq:hat.true} that the asymptotic loss optimality holds.

\underline{(ii) Consider $\phi_n=2$.} Let us first consider the case where $(k_{M_0+1}-k_1)/(n\eta_n)\to 0$ and $M_0/(n\eta_n^2)\to 0$. It follows from \eqref{eq:loss.h.t} that
\begin{align}\label{eq:phi2}
	\frac{L_n(\hat{\mathbf{w}})}{\inf_{\mathbf{w}\in \mathcal{H}_n} L_n(\mathbf{w})}&\geq \frac{L_n(\hat{\mathbf{w}})}{L_n(\mathbf{w}_{M_0+1}^0)} \notag\\
	&=1+o_p(1)+\sum_{m>M_0+1} \left(\hat{w}_m^2+2\hat{w}_m \sum_{l=m+1}^{M_n} \hat{w}_l\right) \left(\frac{\mathbf{e}^{\top}\mathbf{P}_m\mathbf{e}}{\mathbf{e}^{\top}\mathbf{P}_{M_0+1}\mathbf{e}}-1\right) \notag\\
	&\geq 1+o_p(1)+\sum_{m>M_0+1} \hat{w}_m^2 \left(\frac{\mathbf{e}^{\top}\mathbf{P}_m\mathbf{e}}{\mathbf{e}^{\top}\mathbf{P}_{M_0+1}\mathbf{e}}-1\right).
\end{align}
Observe that $\sum_{m>M_0+1} \hat{w}_m^2 (\frac{\mathbf{e}^{\top}\mathbf{P}_m\mathbf{e}}{\mathbf{e}^{\top}\mathbf{P}_{M_0+1}\mathbf{e}}-1)\geq 0$. If $L_n(\hat{\mathbf{w}})/\inf_{\mathbf{w}\in \mathcal{H}_n} L_n(\mathbf{w})\to_p 1$, then from \eqref{eq:phi2}, we can assert that $\sum_{m>M_0+1} \hat{w}_m^2 (\frac{\mathbf{e}^{\top}\mathbf{P}_m\mathbf{e}}{\mathbf{e}^{\top}\mathbf{P}_{M_0+1}\mathbf{e}}-1)\to_p 0$. Therefore, as long as $\sum_{m>M_0+1} \hat{w}_m^2 (\frac{\mathbf{e}^{\top}\mathbf{P}_m\mathbf{e}}{\mathbf{e}^{\top}\mathbf{P}_{M_0+1}\mathbf{e}}-1)\nrightarrow_p 0$, the asymptotic loss optimality does not hold.

Next, we consider the case where $k_{M_0+1}\to \infty$, $(k_{M_0+1}-k_1)/(n\eta_n)\to 0$, $M_0/(k_{M_0+1}\eta_n)\to 0$, and $k_{M_n}/k_{M_0+1}\to 1$. Similar to \eqref{eq:quadform}, it is easy to show that $\mathbf{e}^{\top}\mathbf{P}_{M_n}\mathbf{e}/(\sigma^2 k_{M_n})\to_p 1$. Therefore, $\mathbf{e}^{\top}\mathbf{P}_{M_n}\mathbf{e}/\mathbf{e}^{\top}\mathbf{P}_{M_0+1}\mathbf{e}\to_p k_{M_n}/k_{M_0+1}\to 1$ and
\begin{align*}
	\frac{II_{n,3}}{\mathbf{e}^{\top}\mathbf{P}_{M_0+1}\mathbf{e}}&=\sum_{m>M_0+1} \left(\hat{w}_m^2+2\hat{w}_m \sum_{l=m+1}^{M_n} \hat{w}_l\right) \left(\frac{\mathbf{e}^{\top}\mathbf{P}_m\mathbf{e}}{\mathbf{e}^{\top}\mathbf{P}_{M_0+1}\mathbf{e}}-1\right) \\
	&\leq \left(\sum_{m>M_0+1}\hat{w}_m\right)^2 \left(\frac{\mathbf{e}^{\top}\mathbf{P}_{M_n}\mathbf{e}}{\mathbf{e}^{\top}\mathbf{P}_{M_0+1}\mathbf{e}}-1\right) \\
	&\leq \frac{\mathbf{e}^{\top}\mathbf{P}_{M_n}\mathbf{e}}{\mathbf{e}^{\top}\mathbf{P}_{M_0+1}\mathbf{e}}-1\to_p 0.
\end{align*}
It follows that \eqref{eq:hat.true} holds, which along with Lemma \ref{lem:loss}(ii), yields that the asymptotic loss optimality holds.

\subsection{Proof of the Results in Remark \ref{rem:subgauss}}\label{subsec:subg}

If $e_i$ is sub-Gaussian with a variance proxy $\bar{\sigma}^2$, it can be easily verified that for any $m\in \{1,\ldots,M_0\}$, $\frac{\bm{\mu}^{\top} (\mathbf{P}_{M_0+1}-\mathbf{P}_m)\mathbf{e}}{\sqrt{\bm{\mu}^{\top} (\mathbf{P}_{M_0+1}-\mathbf{P}_m)\bm{\mu}}}$ is also sub-Gaussian with the variance proxy $\bar{\sigma}^2$. It follows from Exercise 2.12 of \citet{wainwright.HDS2019} that
\begin{equation*}
	E \left\{\max_{1\leq m\leq M_0} \left|\frac{\bm{\mu}^{\top} (\mathbf{P}_{M_0+1}-\mathbf{P}_m)\mathbf{e}}{\sqrt{\bm{\mu}^{\top} (\mathbf{P}_{M_0+1}-\mathbf{P}_m)\bm{\mu}}}\right|\right\}\leq \sqrt{2\bar{\sigma}^2 \log(2M_0)},
\end{equation*}
which along with Markov's inequality and Condition \ref{cond:mu}, yields that
\begin{equation}\label{eq:max.mue2}
	\max_{1\leq m\leq M_0} \left|\bm{\mu}^{\top} (\mathbf{P}_{M_0+1}-\mathbf{P}_m)\mathbf{e}\right|=O_p\left(\{n \log(2M_0)\}^{1/2}\right).
\end{equation}
Consequently, if $e_i$ is sub-Gaussian, we can replace \eqref{eq:max.mue} with \eqref{eq:max.mue2}, and thus the condition $M_0/(k_{M_0+1}\eta_n)\to 0$ in Lemma \ref{lem:loss}(ii) can be relaxed to $\log(2M_0)/(k_{M_0+1}\eta_n)\to 0$.

Next, by using \eqref{eq:max.mue2}, it is easy to see that when the conditions $M_0/(k_{M_0+1}\eta_n)\to 0$, $\phi_n^2 M_0 (k_{M_0+1}-k_1)^2/(k_{M_0+1}^2 n\eta_n^2)\to 0$, and $M_0/(n\eta_n^2)\to 0$ in Appendix \ref{subsec:thm1} are replaced by $\log(2M_0)/(k_{M_0+1}\eta_n)\to 0$, $\phi_n^2 \log(2M_0) (k_{M_0+1}-k_1)^2/(k_{M_0+1}^2 n\eta_n^2)\to 0$, and $\log(2M_0)/(n\eta_n^2)\to 0$, respectively, the results in Appendix \ref{subsec:thm1} still hold. Moreover, $\phi_n^2(k_{M_0+1}-k_1)/(n\eta_n)\to 0$ and $\log(2M_0)/(k_{M_0+1}\eta_n)\to 0$ imply that $\phi_n^2 \log(2M_0) (k_{M_0+1}-k_1)^2/(k_{M_0+1}^2 n\eta_n^2)\to 0$. Consequently, the conditions $M_0/(k_{M_0+1}\eta_n)\to 0$ and $M_0/(n\eta_n^2)\to 0$ in Theorem \ref{thm:opt.loss} can be relaxed to $\log(2M_0)/(k_{M_0+1}\eta_n)\to 0$ and $\log(2M_0)/(n\eta_n^2)\to 0$, respectively.

\subsection{Proof of Lemma \ref{lem:risk}}

By the second part of Lemma \ref{lem:lossrisk}, it is easy to show that
\begin{align*}
	&\inf_{\mathbf{w}\in \mathcal{H}_n} R_n(\mathbf{w})-R_n(\mathbf{w}_{M_0+1}^0) \\
	&\quad =\inf_{(w_1,\ldots,w_{M_0})\in \mathcal{U}_n} \Biggl[\sum_{m=1}^{M_0} \left(w_m^2+2w_m\sum_{l=1}^{m-1} w_l\right) \left\{\bm{\mu}^{\top} (\mathbf{P}_{M_0+1}-\mathbf{P}_m)\bm{\mu}+\sigma^2 (k_{M_0+1}-k_m)\right\} \\
	&\hspace{3.3cm} -2\sigma^2 \sum_{m=1}^{M_0} w_m (k_{M_0+1}-k_m)\Biggr] \\
	&\quad \geq \inf_{(w_1,\ldots,w_{M_0})\in \mathcal{U}_n} \Biggl[\left(\sum_{m=1}^{M_0} w_m\right)^2 \left\{\bm{\mu}^{\top} (\mathbf{P}_{M_0+1}-\mathbf{P}_{M_0})\bm{\mu}+\sigma^2 (k_{M_0+1}-k_{M_0})\right\} \\
	&\hspace{3.3cm} -\left(\sum_{m=1}^{M_0} w_m\right) 2\sigma^2(k_{M_0+1}-k_1)\Biggr] \\
	&\quad \geq -\frac{\sigma^4(k_{M_0+1}-k_1)^2}{\bm{\mu}^{\top} (\mathbf{P}_{M_0+1}-\mathbf{P}_{M_0})\bm{\mu}+\sigma^2 (k_{M_0+1}-k_{M_0})} \\
	&\quad \geq -\frac{\sigma^4(k_{M_0+1}-k_1)^2}{\kappa_0 n\eta_n+\sigma^2 (k_{M_0+1}-k_{M_0})},
\end{align*}
where the last inequality follows from Lemma \ref{lem:mumu}. Since $R_n(\mathbf{w}_{M_0+1}^0)=\sigma^2 k_{M_0+1}$ and $(k_{M_0+1}-k_1)/(n\eta_n)\to 0$, we have
\begin{align*}
	1\geq \frac{\inf_{\mathbf{w}\in \mathcal{H}_n} R_n(\mathbf{w})} {R_n(\mathbf{w}_{M_0+1}^0)} &\geq 1-\frac{\sigma^2(k_{M_0+1}-k_1)}{\kappa_0 n\eta_n+\sigma^2 (k_{M_0+1}-k_{M_0})}\times \frac{k_{M_0+1}-k_1}{k_{M_0+1}} \\
	&\geq 1-\frac{\sigma^2(k_{M_0+1}-k_1)/(n\eta_n)}{\kappa_0+\sigma^2 (k_{M_0+1}-k_{M_0})/(n\eta_n)} \to 1.
\end{align*}
This completes the proof of Lemma \ref{lem:risk}.

\subsection{Proof of Theorem \ref{thm:opt.risk}}\label{subsec:thm2}

From the second part of Lemma \ref{lem:lossrisk}, we have $R_n(\hat{\mathbf{w}})=R_n(\mathbf{w}_{M_0+1}^0)+III_{n,1}-2III_{n,2}+III_{n,3}$,
where
\begin{align*}
	& III_{n,1}=\sum_{m=1}^{M_0} \left(\hat{w}_m^2+2\hat{w}_m\sum_{l=1}^{m-1} \hat{w}_l\right) \left\{\bm{\mu}^{\top} (\mathbf{P}_{M_0+1}-\mathbf{P}_m)\bm{\mu}+\sigma^2 (k_{M_0+1}-k_m)\right\}, \\
	& III_{n,2}=\sigma^2 \sum_{m=1}^{M_0} \hat{w}_m (k_{M_0+1}-k_m),
	\intertext{and}
	& III_{n,3}=\sigma^2 \sum_{m>M_0+1} \left(\hat{w}_m^2+2\hat{w}_m \sum_{l=m+1}^{M_n} \hat{w}_l\right) (k_m-k_{M_0+1}).
\end{align*}
Using Lemma \ref{lem:what} and Condition \ref{cond:mu}, we have
\begin{align*}
	III_{n,1}&\leq \left(\sum_{m=1}^{M_0} \hat{w}_m\right)^2 \left\{\bm{\mu}^{\top} (\mathbf{P}_{M_0+1}-\mathbf{P}_1)\bm{\mu}+\sigma^2 (k_{M_0+1}-k_1)\right\} \\
	&\leq O_p\left\{\frac{\phi_n^2 (k_{M_0+1}-k_1)^2}{(n\eta_n)^2}\right\}\{\|\bm{\mu}\|^2+\sigma^2 (k_{M_0+1}-k_1)\} \\
	&=O_p\left\{\frac{\phi_n^2 (k_{M_0+1}-k_1)^2}{n\eta_n^2}\right\}+O_p\left\{\frac{\phi_n^2 (k_{M_0+1}-k_1)^3}{(n\eta_n)^2}\right\}
\end{align*}
and
\begin{equation*}
	III_{n,2}\leq \sigma^2(k_{M_0+1}-k_1) \sum_{m=1}^{M_0} \hat{w}_m=O_p\left\{\frac{\phi_n (k_{M_0+1}-k_1)^2}{n\eta_n}\right\}.
\end{equation*}
Since $R_n(\mathbf{w}_{M_0+1}^0)=\sigma^2 k_{M_0+1}$ and $\phi_n^2 (k_{M_0+1}-k_1)/(n\eta_n^2)\to 0$, we have $III_{n,1}/R_n(\mathbf{w}_{M_0+1}^0)=o_p(1)$ and $III_{n,2}/R_n(\mathbf{w}_{M_0+1}^0)=o_p(1)$. Therefore,
\begin{equation}\label{eq:risk.h.t}
	\frac{R_n(\hat{\mathbf{w}})}{R_n(\mathbf{w}_{M_0+1}^0)}=1+o_p(1)+\frac{III_{n,3}}{\sigma^2 k_{M_0+1}}.
\end{equation}
Next, we consider $\phi_n\to \infty$ and $\phi_n=2$, respectively.

\underline{(i) Consider $\phi_n\to \infty$.} On the set $\mathfrak{F}_n=\{\sum_{m>M_0+1} \hat{w}_m=0\}$, we have $III_{n,3}=0$, which along with \eqref{eq:risk.h.t} and the fact that $\Pr(\mathfrak{F}_n)\to 1$ from Lemma \ref{lem:what}(ii), yields that
\begin{equation}\label{eq:risk.hat.true}
	\frac{R_n(\hat{\mathbf{w}})}{R_n(\mathbf{w}_{M_0+1}^0)}=1+o_p(1).
\end{equation}
Observe that $\phi_n^2 (k_{M_0+1}-k_1)/(n\eta_n^2)\to 0$ implies that $(k_{M_0+1}-k_1)/(n\eta_n)\to 0$. Finally, if $\phi_n^2 (k_{M_0+1}-k_1)/(n\eta_n^2)\to 0$, we can deduce from Lemma \ref{lem:risk} and \eqref{eq:risk.hat.true} that the asymptotic risk optimality holds.

\underline{(ii) Consider $\phi_n=2$.} If $(k_{M_0+1}-k_1)/(n\eta_n^2)\to 0$, it follows from \eqref{eq:risk.h.t} that
\begin{align*}
	\frac{R_n(\hat{\mathbf{w}})}{\inf_{\mathbf{w}\in \mathcal{H}_n} R_n(\mathbf{w})}&\geq \frac{R_n(\hat{\mathbf{w}})}{R_n(\mathbf{w}_{M_0+1}^0)} \\
	&=1+o_p(1)+\sum_{m>M_0+1} \left(\hat{w}_m^2+2\hat{w}_m \sum_{l=m+1}^{M_n} \hat{w}_l\right) \left(\frac{k_m}{k_{M_0+1}}-1\right) \\
	&\geq 1+o_p(1)+\sum_{m>M_0+1} \hat{w}_m^2 \left(\frac{k_m}{k_{M_0+1}}-1\right).
\end{align*}
Observe that $\sum_{m>M_0+1} \hat{w}_m^2 (\frac{k_m}{k_{M_0+1}}-1)\geq 0$. Therefore, if $(k_{M_0+1}-k_1)/(n\eta_n^2)\to 0$, then as long as $\sum_{m>M_0+1} \hat{w}_m^2 (\frac{k_m}{k_{M_0+1}}-1)\nrightarrow_p 0$, the asymptotic risk optimality does not hold.

Next, we consider the case where $k_{M_0+1}\to \infty$, $(k_{M_0+1}-k_1)/(n\eta_n^2)\to 0$, and $k_{M_n}/k_{M_0+1}\to 1$. Observe that
\begin{align*}
	\frac{III_{n,3}}{\sigma^2 k_{M_0+1}}&=\sum_{m>M_0+1} \left(\hat{w}_m^2+2\hat{w}_m \sum_{l=m+1}^{M_n} \hat{w}_l\right)\left(\frac{k_m}{k_{M_0+1}}-1\right) \\
	&\leq \left(\sum_{m>M_0+1}\hat{w}_m\right)^2 \left(\frac{k_{M_n}}{k_{M_0+1}}-1\right) \\
	&\leq \frac{k_{M_n}}{k_{M_0+1}}-1 \to 0.
\end{align*}
It follows that \eqref{eq:risk.hat.true} holds, which along with Lemma \ref{lem:risk}, yields that the asymptotic risk optimality holds.

\newpage

\baselineskip=14pt
\bibliographystyle{apalike}
\bibliography{asyopt}

\clearpage\pagebreak\newpage
\pagestyle{empty}

\begin{table}[htpb]
	\caption{Simulation results for Example \ref{exam:fixdim}. The loss and risk ratios for the true model $\mathbf{w}_{M_0+1}^0$ and the least squares model averaging with $\phi_n=2$ and $\phi_n=\log n$.}\label{tab:simu.exm1}
	\centering
	\begin{tabular}{cccccccc}
			\hline
			& & \multicolumn{3}{c}{Loss ratio to $\inf_{\mathbf{w}\in \mathcal{H}_n} L_n(\mathbf{w})$} & \multicolumn{3}{c}{Risk ratio to $\inf_{\mathbf{w}\in \mathcal{H}_n} R_n(\mathbf{w})$} \\
			\cmidrule(r){3-5}\cmidrule(l){6-8}
			$R^2$ & $n$ & $\mathbf{w}_{M_0+1}^0$ & $\phi_n=2$ & $\phi_n=\log n$ & $\mathbf{w}_{M_0+1}^0$ & $\phi_n=2$ & $\phi_n=\log n$ \\
			\hline
			\multirow{4}{*}{0.05} & 100 & 3.020 & 2.928 & 3.342 & 1.558 & 1.353 & 1.613 \\   
			& 1000 & 2.329 & 2.647 & 2.924 & 1.062 & 1.113 & 1.708 \\   
			& 10000 & 1.843 & 2.330 & 1.755 & 1.006 & 1.080 & 1.085 \\  
			& 50000 & 2.901 & 3.587 & 2.502 & 1.001 & 1.075 & 1.024 \\
			\hline   
			\multirow{4}{*}{0.1} & 100 & 2.392 & 3.126 & 3.887 & 1.287 & 1.361 & 1.922 \\   
			& 1000 & 2.035 & 2.519 & 2.210 & 1.030 & 1.085 & 1.239 \\   
			& 10000 & 2.194 & 2.802 & 1.895 & 1.003 & 1.081 & 1.040 \\   
			& 50000 & 3.375 & 5.847 & 2.975 & 1.001 & 1.080 & 1.012 \\
			\hline   
			\multirow{4}{*}{0.5} & 100 & 2.272 & 2.803 & 2.159 & 1.034 & 1.096 & 1.111 \\   
			& 1000 & 1.986 & 2.657 & 1.864 & 1.003 & 1.081 & 1.022 \\   
			& 10000 & 1.999 & 2.949 & 1.930 & 1.000 & 1.080 & 1.005 \\   
			& 50000 & 1.858 & 2.662 & 1.815 & 1.000 & 1.079 & 1.002 \\
			\hline   
			\multirow{4}{*}{0.9} & 100 & 3.131 & 3.435 & 2.520 & 1.004 & 1.083 & 1.016 \\   
			& 1000 & 2.072 & 2.754 & 2.013 & 1.000 & 1.080 & 1.004 \\   
			& 10000 & 1.826 & 2.430 & 1.809 & 1.000 & 1.081 & 1.001 \\   
			& 50000 & 2.260 & 2.892 & 2.236 & 1.000 & 1.081 & 1.000 \\ 
			\hline
	\end{tabular}
\end{table}

\clearpage\pagebreak\newpage
\pagestyle{empty}

\begin{figure}[htpb]
	\centering
	\includegraphics[scale=0.54]{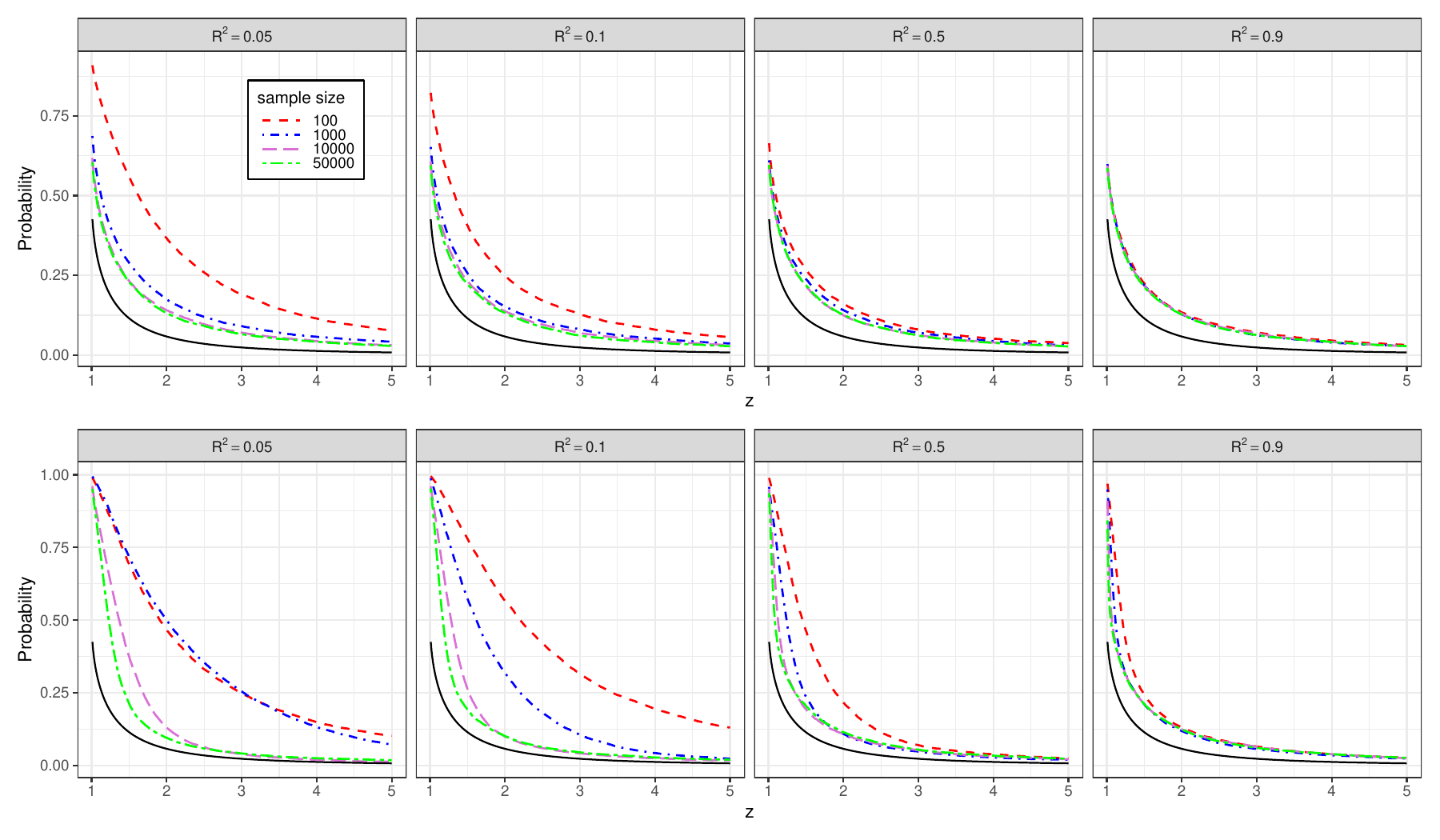}
	\caption{Simulation results for Example \ref{exam:fixdim}, where $\rho=0$. Upper panel: the simulated value of $\Pr\{L_n(\mathbf{w}_{M_0+1}^0)/L_n(\mathbf{w}^L)\geq z\}$ is plotted against $z\in (0, 5]$. Lower panel: the simulated value of $\Pr\{L_n(\hat{\mathbf{w}})/L_n(\mathbf{w}^L)\geq z\}$ is plotted against $z$. The solid line is the function $2^{-1}\Pr\{\mathrm{Beta}(1/2, 2)\geq z\}$.}\label{fig:prob.ineq}
\end{figure}

\clearpage\pagebreak\newpage
\pagestyle{empty}

\begin{table}[htpb]
	\caption{Simulation results for Example \ref{exam:divdim1}. The loss and risk ratios for the true model $\mathbf{w}_{M_0+1}^0$ and the least squares model averaging with $\phi_n=2$ and $\phi_n=\log n$.}\label{tab:simu.exm2}
	\centering
	\begin{tabular}{cccccccc}
		\hline
		& & \multicolumn{3}{c}{Loss ratio to $\inf_{\mathbf{w}\in \mathcal{H}_n} L_n(\mathbf{w})$} & \multicolumn{3}{c}{Risk ratio to $\inf_{\mathbf{w}\in \mathcal{H}_n} R_n(\mathbf{w})$} \\
		\cmidrule(r){3-5}\cmidrule(l){6-8}
		$R^2$ & $n$ & $\mathbf{w}_{M_0+1}^0$ & $\phi_n=2$ & $\phi_n=\log n$ & $\mathbf{w}_{M_0+1}^0$ & $\phi_n=2$ & $\phi_n=\log n$ \\
		\hline
		\multirow{5}{*}{0.05} & 100 & 1.442 & 1.327 & 1.165 & 1.294 & 1.094 & 1.060 \\   
		& 1000 & 1.069 & 1.142 & 1.330 & 1.031 & 1.049 & 1.286 \\   
		& 10000 & 1.020 & 1.045 & 1.052 & 1.002 & 1.008 & 1.024 \\   
		& 50000 & 1.010 & 1.024 & 1.016 & 1.000 & 1.005 & 1.004 \\   
		& 100000 & 1.008 & 1.019 & 1.011 & 1.000 & 1.004 & 1.002 \\ 
		\hline  
		\multirow{5}{*}{0.1} & 100 & 1.336 & 1.356 & 1.288 & 1.203 & 1.120 & 1.150 \\   
		& 1000 & 1.061 & 1.127 & 1.244 & 1.016 & 1.025 & 1.176 \\   
		& 10000 & 1.019 & 1.044 & 1.034 & 1.001 & 1.008 & 1.011 \\   
		& 50000 & 1.011 & 1.024 & 1.013 & 1.000 & 1.005 & 1.002 \\   
		& 100000 & 1.008 & 1.019 & 1.009 & 1.000 & 1.004 & 1.001 \\
		\hline   
		\multirow{5}{*}{0.5} & 100 & 1.180 & 1.371 & 1.381 & 1.035 & 1.063 & 1.158 \\   
		& 1000 & 1.049 & 1.111 & 1.071 & 1.002 & 1.019 & 1.012 \\   
		& 10000 & 1.018 & 1.043 & 1.021 & 1.000 & 1.008 & 1.001 \\   
		& 50000 & 1.011 & 1.024 & 1.011 & 1.000 & 1.005 & 1.000 \\   
		& 100000 & 1.008 & 1.019 & 1.008 & 1.000 & 1.004 & 1.000 \\
		\hline   
		\multirow{5}{*}{0.9} & 100 & 1.151 & 1.325 & 1.192 & 1.004 & 1.047 & 1.013 \\   
		& 1000 & 1.049 & 1.109 & 1.053 & 1.000 & 1.019 & 1.002 \\   
		& 10000 & 1.019 & 1.043 & 1.019 & 1.000 & 1.008 & 1.000 \\   
		& 50000 & 1.010 & 1.024 & 1.010 & 1.000 & 1.005 & 1.000 \\   
		& 100000 & 1.008 & 1.019 & 1.008 & 1.000 & 1.004 & 1.000 \\ 
		\hline
	\end{tabular}
\end{table}

\clearpage\pagebreak\newpage
\pagestyle{empty}

\begin{table}[htpb]
	\caption{Simulation results for Example \ref{exam:divdim2}. The loss and risk ratios for the true model $\mathbf{w}_{M_0+1}^0$ and the least squares model averaging with $\phi_n=2$ and $\phi_n=\log n$.}\label{tab:simu.exm3}
	\centering
	\begin{tabular}{cccccccc}
		\hline
		& & \multicolumn{3}{c}{Loss ratio to $\inf_{\mathbf{w}\in \mathcal{H}_n} L_n(\mathbf{w})$} & \multicolumn{3}{c}{Risk ratio to $\inf_{\mathbf{w}\in \mathcal{H}_n} R_n(\mathbf{w})$} \\
		\cmidrule(r){3-5}\cmidrule(l){6-8}
		$R^2$ & $n$ & $\mathbf{w}_{M_0+1}^0$ & $\phi_n=2$ & $\phi_n=\log n$ & $\mathbf{w}_{M_0+1}^0$ & $\phi_n=2$ & $\phi_n=\log n$ \\
		\hline
		\multirow{5}{*}{0.05} & 100 & 5.833 & 13.518 & 7.295 & 1.462 & 1.371 & 1.633 \\   
		& 1000 & 1.518 & 1.648 & 2.771 & 1.097 & 1.058 & 1.920 \\   
		& 10000 & 1.207 & 1.306 & 1.524 & 1.017 & 1.027 & 1.237 \\   
		& 100000 & 1.137 & 1.221 & 1.192 & 1.002 & 1.021 & 1.053 \\   
		& 500000 & 1.106 & 1.177 & 1.120 & 1.001 & 1.018 & 1.018 \\
		\hline
		\multirow{5}{*}{0.1} & 100 & 9.398 & 20.300 & 18.442 & 1.238 & 1.350 & 1.927 \\   
		& 1000 & 1.489 & 1.635 & 1.959 & 1.046 & 1.044 & 1.348 \\   
		& 10000 & 1.201 & 1.304 & 1.319 & 1.008 & 1.027 & 1.109 \\   
		& 100000 & 1.133 & 1.218 & 1.151 & 1.001 & 1.022 & 1.025 \\   
		& 500000 & 1.102 & 1.173 & 1.107 & 1.000 & 1.018 & 1.008 \\
		\hline
		\multirow{5}{*}{0.5} & 100 & 5.836 & 6.115 & 3.816 & 1.028 & 1.066 & 1.081 \\
		& 1000 & 1.365 & 1.568 & 1.371 & 1.005 & 1.038 & 1.034 \\   
		& 10000 & 1.180 & 1.294 & 1.181 & 1.001 & 1.026 & 1.012 \\   
		& 100000 & 1.127 & 1.217 & 1.126 & 1.000 & 1.022 & 1.003 \\   
		& 500000 & 1.100 & 1.173 & 1.100 & 1.000 & 1.019 & 1.001 \\
		\hline   
		\multirow{5}{*}{0.9} & 100 & 8.618 & 7.585 & 5.133 & 1.003 & 1.061 & 1.015 \\   
		& 1000 & 1.460 & 1.680 & 1.432 & 1.001 & 1.039 & 1.005 \\   
		& 10000 & 1.168 & 1.288 & 1.167 & 1.000 & 1.026 & 1.002 \\   
		& 100000 & 1.126 & 1.217 & 1.126 & 1.000 & 1.021 & 1.000 \\   
		& 500000 & 1.101 & 1.173 & 1.100 & 1.000 & 1.018 & 1.000 \\ 
		\hline
	\end{tabular}
\end{table}

\end{document}